\documentclass[matprg]{svjour3}

\usepackage{amsfonts}
\usepackage{latexsym}
\usepackage{amsmath, amssymb}

\usepackage{algorithm}\usepackage{algpseudocode}

\newcommand{\tsum}{\textstyle{\sum}}

\newcommand{\beq}{\begin{equation}}
\newcommand{\eeq}{\end{equation}}
\newcommand{\beqa}{\begin{eqnarray}}
\newcommand{\eeqa}{\end{eqnarray}}
\newcommand{\beqas}{\begin{eqnarray*}}
\newcommand{\eeqas}{\end{eqnarray*}}
\newcommand{\bi}{\begin{itemize}}
\newcommand{\ei}{\end{itemize}}
\newcommand{\ba}{\begin{array}}
\newcommand{\ea}{\end{array}}
\newcommand{\gap}{\hspace*{2em}}

\newcommand{\nn}{\nonumber}

\def\endproof{{\ \hfill\hbox{%
      \vrule width1.0ex height1.0ex
    }\parfillskip 0pt}\par}

\def\eqnok#1{(\ref{#1})}
\def\argmin{{\rm argmin}}

\def\Argmin{{\rm Argmin}}

\def\exp{{\rm exp}}

\def\vgap{\vspace*{.1in}}
\def\endproof{{\ \hfill\hbox{%
      \vrule width1.0ex height1.0ex
    }\parfillskip 0pt}\par}

\def\all1{{{\bf e} }}

\def\cX{{h}}

\newcommand{\bbe}{\mathbb{E}}
\def\prob{\mathop{\rm Prob}}
\def\Prob{{\hbox{\rm Prob}}}
\newcommand{\bbr}{\mathbb{R}}

\def\w{\omega}

\def\Pr{{{\cal M}}}
\def\rec{{R_{ac}}}
\def\cQ{{\cal Q}}
\def\gap{{\rm gap}}

\title{
An optimal randomized incremental gradient method 
\thanks{The author of this paper was partially supported by
    NSF grant  CMMI-1537414, DMS-1319050, 
    ONR grant N00014-13-1-0036 and
    NSF CAREER Award CMMI-1254446.}}
\author{
    Guanghui Lan
    \thanks{Department of Industrial and Systems
    Engineering, University of Florida, Gainesville, FL, 32611.
    (email: {\tt glan@ise.ufl.edu}).}
    \and
    Yi Zhou
    \thanks{Department of Industrial and Systems
    Engineering, University of Florida, Gainesville, FL, 32611.
    (email: {\tt yizhou@ufl.edu}).}
}

\begin{document}

\maketitle

\begin{abstract}
In this paper, we consider a class of finite-sum convex optimization problems whose objective function is given by 
the summation of $m$ ($\ge 1$) smooth components together with some other relatively simple
terms. We first introduce a deterministic primal-dual gradient (PDG) method that can achieve the optimal black-box iteration complexity 
for solving these composite optimization problems
using a primal-dual termination criterion. Our major contribution is to develop
a randomized primal-dual gradient (RPDG) method, which
needs to compute the gradient of only one randomly selected smooth component at each iteration,
but can possibly achieve better complexity than PDG in terms of the total number of gradient evaluations. 
More specifically, we show that the total number of gradient evaluations performed by RPDG
can be ${\cal O} (\sqrt{m})$ times smaller, both in expectation and
with high probability, than those performed by deterministic optimal first-order methods under favorable situations.
We also show that the complexity of the RPDG method is not improvable 
by developing a new lower complexity bound for a general class of randomized methods for solving
large-scale finite-sum convex optimization problems.
Moreover, through the development of PDG and RPDG, we introduce a novel game-theoretic interpretation
for these optimal methods for convex optimization.


\vspace{.1in}


\vspace{.07in}


\end{abstract}

\vspace{0.1cm}

\setcounter{equation}{0}
\section{Introduction}
The basic problem of interest in this paper is the convex programming (CP) problem given by
\beq \label{cp}
\Psi^* := \min_{x \in X} \left\{\Psi(x) := \tsum_{i=1}^m f_i(x) + h(x) + \mu \, \w(x) \right\}.
\eeq
Here, $X \subseteq \bbr^n$ is a closed convex set, $\cX$ is a relatively simple convex function,
$f_i: \bbr^n \to \bbr$, $i = 1, \ldots, m$, are smooth convex functions with Lipschitz
continuous gradient, i.e., $ \exists L_i \ge 0$ such that
\beq \label{smoothness}
\| \nabla f_i(x_1) - \nabla f_i(x_2) \|_* \le L_i \|x_1 - x_2 \|, \ \ \forall x_1, x_2 \in \bbr^n,
\eeq
$\w: X \to \bbr$ is a strongly convex function with modulus $1$ w.r.t. an arbitrary norm $\|\cdot\|$, i.e.,
\begin{equation}\label{s27}
 \langle \w'(x_1)-\w'(x_2), x_1-x_2 \rangle \geq \tfrac{1}{2} \|x_1-x_2\|^2,\;\;\forall x_1,x_2 \in X,
\end{equation}
and $\mu \ge 0$ is a given constant. Hence,  the objective function $\Psi$ is strongly convex whenever $\mu > 0$.
For notational convenience, we also denote $f(x) \equiv \tsum_{i=1}^m f_i(x)$ and
$L \equiv \tsum_{i=1}^m L_i$. 
It is easy to see that for some $L_f \ge 0$,
\beq \label{smooth_f}
\|\nabla f(x_1) - \nabla f(x_2) \|_* \le L_f \|x_1 - x_2\| \le L \|x_1- x_2\|, \ \ \forall x_1, x_2 \in \bbr^n.
\eeq
 Throughout this paper, we assume subproblems of the form
\beq \label{simple_problem}
\argmin_{x \in X} \langle g, x\rangle + h(x) + \mu \, \w(x)
\eeq
are easy to solve.
CP given in the form of \eqnok{cp} has recently found a wide range of applications
in machine learning, statistics, and image processing, and hence becomes the subject of 
intensive studies during the past few years.

Stochastic (sub)gradient descent (SGD) (a.k.a. stochastic approximation (SA)) type methods have been proven 
useful to solve problems given in the form of~\eqnok{cp}.
SGD was originally designed to solve stochastic optimization problems given by
\beq \label{SP}
\min_{x \in X} \bbe_\xi[F(x, \xi)],
\eeq
where $\xi$ is a random variable with support $\Xi \subseteq \bbr^d$. Problem \eqnok{cp} can be viewed as
a special case of \eqnok{SP} by setting $\xi$ to be a discrete random variable supported on $\{1, \ldots, m\}$ with
$\Prob\{\xi = i\} = \nu_i$ and $F(x, i) = \nu_i^{-1} f_i(x) + h(x) + \mu \w(x)$, $i = 1, \ldots, m$.
Since each iteration of SGDs needs to compute the (sub)gradient of only one randomly selected $f_i$ \footnote{
Observe that the subgradients of $h$ and $\w$ are not required due to the assumption in \eqnok{simple_problem}.},
their iteration cost is significantly smaller than that for
deterministic first-order methods (FOM), which involves the computation
of first-order information of $f$ and thus all the $m$ (sub)gradients of $f_i$'s. Moreover, when $f_i$'s are 
general nonsmooth convex functions, by properly
specifying the probabilities $\nu_i$, $i=1, \ldots, m$ \footnote{Suppose that $f_i$ are Lipschitz continuous 
with constants $M_i$ and let us denote $M := \tsum_{i=1}^m M_i$, we should set $\nu_i = M_i/M$ in order to get the
optimal complexity for SGDs.}, it can be shown (see \cite{NJLS09-1}) that the iteration complexities for both SGD and FOM are
in the same order of magnitude.
Consequently, the total number of subgradients required by SGDs can be $m$ times smaller than
those by FOMs.

Note however, that there is a significant gap on the complexity bounds between SGDs and deterministic FOMs 
if $f_i$'s are smooth convex functions. For the sake of simplicity, let us focus on the
strongly convex case when $\mu > 0$ and let $x^*$ be the optimal solution of \eqnok{cp}.
In order to find a solution $\bar x \in X$ s.t. $\|\bar x - x^*\|^2 \le \epsilon$, 
the total number of gradient evaluations for $f_i$'s performed by optimal FOMs
can be bounded by
\beq \label{FOM_bounds}
{\cal O} \left\{m \sqrt{\tfrac{L}{\mu}} \log \tfrac{1}{\epsilon}\right\},
\eeq
which was first achieved by the well-known Nesterov's accelerated
gradient method~\cite{Nest83-1,Nest04}, see also relevant extensions in~\cite{Nest13-1,BecTeb09-2,tseng08-1}. On the other hand,
a direct application of optimal SGDs
to the aforementioned stochastic optimization reformulation of \eqnok{cp} would yield an
\beq \label{SA_bounds}
{\cal O} \left\{\sqrt{\tfrac{L}{\mu}} \log \tfrac{1}{\epsilon} + \tfrac{\sigma^2}{\mu \epsilon} \right\}
\eeq
iteration complexity bound on the number of gradient evaluations for $f_i$'s, which was first
achieved by the accelerated stochastic approximation method (\cite{Lan10-3,GhaLan12-2a,GhaLan13-1}).
Here $\sigma >0$ denotes variance of the stochastic gradients.
Clearly, the latter bound is significantly better than the one in \eqnok{FOM_bounds} in terms of its dependence on $m$, but much worse
in terms of its dependence on accuracy $\epsilon$ and a few other problem parameters (e.g., $L$ and $\mu$). 

It should be noted that the optimality of \eqnok{SA_bounds} for general stochastic programming \eqnok{SP}
does not preclude the existence of more efficient algorithms for solving \eqnok{cp},
because \eqnok{cp} is a special case of \eqnok{SP} with finite support $\Xi$.
Last few years have seen very active and fruitful research
in this field (e.g., ~\cite{SchRouBac13-1,JohnsonZhang13-1,DefBacLac14-1,ShaZhang15-1,Yuchen14}).  
In particular, 
Schmidt, Roux and Bach~\cite{SchRouBac13-1} 
presented a stochastic average gradient (SAG) method, which recursively computes an estimator 
of $\nabla f$ by aggregating the gradient of a randomly selected $f_i$ with some 
other previously computed gradient information. They proved that the complexity of SAG is bounded by 
${\cal O}\left( (m + L/\mu ) \log \tfrac{1}{\epsilon}\right)$, see also
Johnson and Zhang~\cite{JohnsonZhang13-1} and Defazio et al.~\cite{DefBacLac14-1}
for similar complexity results for solving \eqnok{cp}. 
In a related but different line of research, Shalev-Shwartz and Zhang~\cite{ShaZhang15-1}
studied a special class of CP problems given in the form of \eqnok{cp} with $f_i(x)$ given by $\phi_i(a_i^T x)$,
where $a_i$ denotes an affine mapping. Under the assumption that
$\w(x) = \|x\|^2_2$, 
they presented an accelerated stochastic dual coordinate ascent (A-SDCA) method,
obtained by properly restarting a stochastic coordinate ascent method
in \cite{ShaZhang13-1} applied to the dual of \eqnok{cp}. Shalev-Shwartz and Zhang show that the iteration complexity 
of this method can be bounded by 
$
{\cal O} \left\{ \left(m + \sqrt{ \tfrac{m L}{\mu}}\right)   \log \tfrac{1}{\epsilon}\right\}.
$
However,
each iteration of A-SDCA requires, instead of the computation of $\nabla f_i$,
the solution of a subproblem given in the form of 
\beq \label{ShaZhang_sub}
\argmin \{ \langle g, y\rangle + \phi_i^*(y)
+ \|y\|_*^2\},
\eeq
 where $\phi_i^*$ denotes the conjugate function of $\phi_i$.
Moreover, these methods were also 
designed for solving a more special class of problems than \eqnok{cp}.
More recently, Lin, Lu, and Xiao~\cite{LinLuXiao14-1} proposed to 
apply the accelerated coordinate descent methods by Nesterov~\cite{Nest10-1},
and Fercoq and  Richt\'{a}rik’s~\cite{fr13} to obtain similar results for solving
these ``regularized empirical loss functions'' as in \cite{ShaZhang15-1}. 
Zhang and Xiao~\cite{Yuchen14} had also obtained similar results
by using different stochastic primal-dual coordinate decomposition techniques. 

In this paper, we focus on randomized incremental gradient methods that can access the first-order 
information of only one randomly selected smooth component $f_i$ at each iteration (see Bertsekas~\cite{Bertsekas10-1}
for an introduction to incremental gradient methods).
It should be noted that while the algorithms in \cite{SchRouBac13-1,JohnsonZhang13-1,DefBacLac14-1}
belong to incremental gradient methods,
generally speaking, the dual coordinate algorithms in \cite{LinLuXiao14-1,ShaZhang15-1,Yuchen14} 
cannot be considered as incremental gradient methods because they require the solutions of a different subproblem
rather than the computation of the gradient of $f_i$.
The previous attempts to improve the complexity of the existing incremental gradient methods, e.g., based on the extrapolation
idea in Nesterov~\cite{Nest83-1}, however, turned out to be tricky and unsuccessful,
see Section 1.2 of Bertsekas~\cite{Bertsekas10-1}  and Section 5 of Agarwal and Bottou~\cite{AgrBott14-1}
for more discussions. Another important yet unresolved issue is that there does not exist a
valid lower complexity bound for randomized incremental gradient methods in the literature.
Hence, it remains unknown what would be the best possible performance that one can expect for 
these types of methods. 
Regarding this question, 
Agarwal and Bottou~\cite{AgrBott14-1}
recently suggested a lower complexity bound for solving problems given in the form of \eqnok{cp}. However,
as pointed out by them in a recent ISMP talk in 2015, the lower complexity bound in \cite{AgrBott14-1}
is deterministic by construction, and hence cannot be used to justify the optimality or suboptimality for the randomized 
incremental gradient methods in \cite{SchRouBac13-1,JohnsonZhang13-1,DefBacLac14-1}
or dual coordinate methods in \cite{LinLuXiao14-1,ShaZhang15-1,Yuchen14}.


Our contribution in this paper mainly lies on the following several aspects. 
Firstly, we present a new class of deterministic FOMs, referred to as the primal-dual gradient (PDG) methods,
which can achieve the optimal black-box iteration complexity in \eqnok{FOM_bounds}
for solving \eqnok{cp}.
The novelty of these methods exists in: 1) a proper reformulation of \eqnok{cp} as a primal-dual saddle point problem
and 2) the incorporation of a new non-differentiable prox-function (or Bregman distance)
based on the conjugate functions of $f_i$ in the dual space. As a consequence, we are able to show that
the PDG method covers a variant of the well-known Nesterov's
accelerated gradient method as a special case.
In particular, the computation of the gradient at the extrapolation point
of the accelerated gradient method
is equivalent to a primal prediction step combined with a dual ascent step (employed
with the aforementioned dual prox-function) in the PDG method.
While it is often difficult to interpret Nesterov's method, 
the development of the PDG method allows us to view this method
as a natural iterative buyer-supplier game. Such a game-theoretic view of 
the accelerated gradient method seems to be new in the literature.
In fact, the obtained complexity results for the PDG method are
slightly stronger than the one in \eqnok{FOM_bounds} and
those in  \cite{Nest83-1,Nest04} for Nesterov's accelerated gradient method, because a stronger 
primal-dual termination criterion has been used in our analysis.

Secondly, we develop
a randomized primal-dual gradient (RPDG) method, which is an incremental
gradient method using only one randomly selected component $\nabla f_i$ at each iteration.
A variant of PDG, this algorithm incorporates an additional dual prediction step
before performing the primal descent step (with a properly defined primal prox-function).
We prove that the number of iterations (and hence the number of gradients) required
by RPDG is bounded by
\beq \label{RPDG_bnd}
{\cal O} \left(\left(m + \sqrt{\tfrac{m L}{\mu}}\right) \log \tfrac{1}{\epsilon} \right),
\eeq
both in expectation and with high probability. 
The complexity bounds of the RPDG method are established in terms of not only the distance from
the iterate $x^k$ to the optimal solution, but also the primal optimality gap based on the 
ergodic mean of the iterates.
In comparison with
the accelerated stochastic dual coordinate ascent method in \cite{ShaZhang15-1}, RPDG deals with
a wider class of problems and can be applied to the cases when $f_i$'s involve a more
complicated composite structure (see examples in \cite{Bertsekas10-1}) and/or a more general regularization term
$\w$ that is strongly convex with respect to an arbitrary norm (see open problems in Section 7 of \cite{ShaZhang15-1}).
Moreover, each iteration of RPDG only involves the computation $\nabla f_i$, rather than
the more complicated subproblem in \eqnok{ShaZhang_sub}, which sometimes may not have explicit
solutions \cite{ShaZhang15-1} (e.g., the logistics regression problem). The RPDG method also admits an interesting game theoretic interpretation, 
implying that by properly incorporating
randomization, the buyer and supplier can reach the equilibrium with possibly fewer price changes
at the expense of more order transactions. 

Thirdly, we show that the number of gradient evaluations required by any randomized incremental gradient
methods to find an $\epsilon$-solution of \eqnok{cp}, i.e., a point $\bar x \in X$ s.t. 
$\bbe[\|\bar x - x^*\|^2_2] \le \epsilon$, cannot be smaller than 
\beq \label{RIG_lb}
{\Omega} \left( \left(m + \sqrt{\tfrac{m L}{\mu}}\right) \log \tfrac{1}{\epsilon}\right),
\eeq
whenever the dimension $n$ is sufficiently large.
This bound is obtained by carefully constructing a special class of separable quadratic programming problems
and tightly bounding the expected distance to the optimal solution for any arbitrary 
distribution used to choose $f_i$ at each iteration.
Comparing \eqnok{RPDG_bnd} with \eqnok{RIG_lb}, we conclude 
that the complexity of the RPDG method is optimal if $n$ is large enough.
To the best of our knowledge, this is the first time that such a lower complexity bound
has been presented for randomized incremental gradient methods in the literature.
As a byproduct, 
we also derived a lower complexity bound for 
randomized block coordinate descent methods by
utilizing the separable structure of the aforementioned worst-case instances. These methods have been  
intensively studied recently, but a valid lower complexity bound is still missing in the literature.


Finally, we generalize RPDG for problems which are not necessarily strongly convex (i.e., $\mu = 0$)
and/or involve structured nonsmooth terms $f_i$. We show that for all these cases, the RPDG
can save ${\cal O}(\sqrt{m})$ times gradient computations (up to certain logarithmic factors)
in comparison with the corresponding optimal deterministic FOMs. In particular, we show that when 
both the primal and dual of \eqnok{cp} are not strongly convex, 
the total number of iterations performed by the RPDG method can be bounded by ${\cal O}(\sqrt{m}/\epsilon)$ (up to 
some logarithmic factors), which is ${\cal O}(\sqrt{m})$ times better, in terms of the total number of dual
subproblems to be solved, than Nesterov's smoothing technique~\cite{Nest05-1}, Nemirovski's mirror-prox
method~\cite{Nem05-1}, or Chambolle and Pock's primal-dual method  \cite{ChamPoc11-1}.
It seems that this complexity result has not been obtained before in the literature. 

It is worth mentioning a few relevant works to our development. The most two related ones are 
conducted independently by Dang and Lan~\cite{DangLan14-1}, and 
Zhang and Xiao~\cite{Yuchen14}.
Both of these papers deal with randomized variants of the primal-dual method presented by
Chambolle and Pock \cite{ChamPoc11-1} (see also extensions in \cite{CheLanOu13-1}) for solving
saddle point problems. 
Zhang and Xiao's development~\cite{Yuchen14} was based on a variant of the primal-dual method 
for solving strongly convex saddle point problems~\cite{ChamPoc11-1}. 
They were able to
show that a block-wise randomized version of the algorithm can achieve similar complexity as
the A-SDCA method in \cite{ShaZhang15-1}. Since Zhang and Xiao's algorithm targets for 
solving a similar class of problems and requires
the solutions of a similar subproblem to \cite{ShaZhang15-1}, it appears that
the aforementioned possible advantages of RPDG over A-SDCA
are also applicable to the stochastic primal-dual coordinate method in \cite{Yuchen14}.
Moreover, the complexity bound of Zhang and Xiao's algorithm is only established in terms of the
Euclidean distances of the iterate $x^k$, $y^k$ to the optimal solution. They did not deal with
the convergence of the ergodic mean of iterates. 
On the other hand,
Dang and Lan's work was motivated by the observation in~\cite{ChenHeYeYuan13-1} that
a direct extension of the alternating direction method of multiplier (ADMM) does not converge for multi-block problems.
Their work in \cite{DangLan14-1} then focuses on the non-strongly convex case and
shows that a randomized primal-dual method, which is equivalent
to a randomized pre-conditioned ADMM for linear constrained problems, does converge for multi-block problems. 
Without incorporating the aforementioned
dual prediction step, the complexity obtained in \cite{DangLan14-1} is ${\cal O}(\sqrt{m})$
times worse than Chambolle and Pock's method. Nevertheless, this is the first time that
randomized algorithms for saddle point optimization with an ${\cal O}(1/\epsilon)$ complexity 
has been presented in the literature. 
More recently, close to the end of the preparation of this paper, we notice that Lin, Mairal, and Harchaoui~\cite{LinMaiHar15-1} 
in a concurrent work presented a catalyst scheme that can be used to accelerate the SAG method in \cite{SchRouBac13-1} and
thus possibly achieve the complexity bound in \eqnok{RPDG_bnd} (under the
Euclidean setting). 
While their approach is an indirect one obtained by properly restarting SAG (or other ``non-accelerated'' first-order methods), 
the proposed randomized primal-dual gradient method is a direct approach with a ``built-in''
acceleration. Also none of these works~\cite{DangLan14-1,Yuchen14,LinMaiHar15-1}
discussed the lower complexity bound for randomized methods.

This paper is organized as follows. We first study the deterministic
primal-dual method in Section~\ref{sec_det}. Section~\ref{sec_ORIG}
is devoted to the design and analysis of the randomized primal-dual method for the strongly convex case,
as well as the development of the lower complexity bound in \eqnok{RIG_lb}.
In Section~\ref{sec_general}, we generalize the RPDG method to different classes of CP problems
that are not necessarily strongly convex. Important technical results and proofs of the main theorems
in Sections~\ref{sec_det} and~\ref{sec_ORIG} are provided in Section~\ref{sec_analysis}. Some brief concluding remarks are made
in Section~\ref{sec_remark}.

\vgap

\noindent {\bf Notation and terminology.}
We use $\|\cdot\|$ to denote an arbitrary norm in $\bbr^n$,
which is not necessarily associated with the inner product $\langle \cdot, \cdot \rangle$.
We also use $\|\cdot\|_*$ to denote the conjugate norm of $\|\cdot\|$.
For any convex function $h$, $\partial h(x)$ is the set of subdifferential at $x$.
Given any $X \subseteq \bbr^n$, we say a convex function $h:X \to \bbr$ is nonsmooth
if $|h(x) - h(y)| \le M_h \|x - y\|$ for any $x, y \in X$. 
We say that a convex function $f: X \to \bbr$ is smooth if it
is Lipschitz continuously differentiable with Lipschitz constant $L>0$, i.e.,
$
\|\nabla f(y) - \nabla f(x)\|_* \le L\|y-x\|$
for any $x, y \in X$.
For any $p \ge 1$, $\|\cdot\|_p$ denotes the standard $p$-norm in  $\bbr^n$, i.e.,
\[
\|x\|_p^p =  \sum_{i=1}^n |x_i|^p, \qquad \mbox{for any } x \in \bbr^n.
\]
For any real number $r$, $\lceil r \rceil$ and $\lfloor r \rfloor$ denote the nearest integer to
$r$ from above and below, respectively. $\bbr_+$ and $\bbr_{++}$, respectively, denote the set of nonnegative 
and positive real numbers. ${\cal N}$ denotes the set of natural numbers $\{1, 2, \ldots\}$.

\setcounter{equation}{0}
\section{An optimal primal-dual gradient method} \label{sec_det}
Our goal in this section is to present a novel primal-dual gradient (PDG)
method for solving \eqnok{cp},  which will also provide a basis for the development of
the randomized primal-dual gradient methods in later sections.
We establish the optimal convergence of this algorithm in terms of
the primal-dual optimality gap under the assumption that
the gradient of $f$ is computed at each iteration.
We show that PDG generalizes one variant of the well-known
Nesterov's accelerated gradient method, and allows a natural game
interpretation,  and hence 
that the latter algorithm also admits a similar interpretation.

\subsection{Preliminaries:  primal and dual prox-functions}\label{sec_prelim}
In this subsection, we discuss both primal and dual prox-functions
(proximity control functions) in the primal and dual spaces, respectively.

Recall that the function $\w:\,X\to \bbr$ in \eqnok{cp} is strongly convex with modulus $1$
with respect to $\|\cdot\|$.  We can define 
a primal {\em prox-function} associated with $\w$ as
\begin{equation}\label{s450}
P(x^0,x) \equiv P_\w(x^0, x) := \w(x) - [\w(x^0)+ \langle \w'(x^0), x-x^0 \rangle],
\end{equation}
where $\w' (x^0) \in \partial \w(x^0)$ is an arbitrary 
subgradient of $\w$ at $x^0$. 
Clearly, by the strong convexity of $\w$, we have
\beq \label{strong_h}
P(x^0, x) \ge \tfrac{1}{2} \|x - x^0\|^2, \ \ \ \forall x, x^0 \in X.
\eeq
Note that the prox-function $P(\cdot,\cdot)$ described above
generalizes the Bregman's distance in the sense that
$\w$ is not necessarily differentiable (see \cite{Breg67,AuTe06-1,BBC03-1,Kiw97-1} and references therein).
Throughout this paper, we assume that the prox-mapping associated with $X$, $\w$, and $h$, 
given by
\beq \label{prox_mapping}
\Pr_X(g, x^0, \eta) \equiv \Pr_{X,\w,h}(g, x^0, \eta) := \arg\min\limits_{x \in X} \left\{ \langle g, x \rangle
 + h(x) + \mu \, \w(x) + \eta P(x^0, x) \right\},
\eeq
is easily computable for any $x^0 \in X, g \in \bbr^n$, $\mu \ge 0$, and $\eta >0$.
Clearly this is equivalent to the assumption that \eqnok{simple_problem} is easy to solve.
Whenever $\w$ is non-differentiable, we need to specify a
particular selection of the subgradient $\w'$ before performing the prox-mapping.
We assume throughout this paper that such a selection of $\w'$ is 
defined recursively as follows. Denote $x^1 \equiv \Pr_X(g, x^0, \eta)$.
By the optimality condition of \eqnok{prox_mapping}, 
we have 
\[
g + h'(x^1) + (\mu + \eta) \w'(x^1) -\eta \w'(x^0) \in {\cal N}_X(x^1),
\]
where ${\cal N}_X$ denotes the normal cone of $X$ at $x^1$.
Once such a $\w'(x^1)$ satisfying the above relation
is identified, we will use it as a subgradient when defining
$P(x^1, x)$ in the next iteration. 

Now let us consider the dual space ${\cal G}$, where the
gradients of $f$ reside, and equip it with the conjugate norm $\|\cdot\|_*$. 
Let $J_f: {\cal G} \to \bbr$ be the conjugate function of $f$ such that
\beq \label{f_conjugate_f}
f(x) := \max_{g \in {\cal G}} \langle x, g \rangle - J_f(g).
\eeq
It is clear that $J_f$ is strongly convex
with modulus $1/ L_f$ w.r.t. $\|\cdot\|_*$. Therefore, we 
can define its associated dual prox-functions and dual prox-mappings as
\begin{align} 
D_f(g^0, g) &:= J_f(g) - [J_f(g^0) + \langle J_f'(g^0), g - g^0 \rangle], \label{def_Df}\\
\Pr_{{\cal G}} (-\tilde x, g^0, \tau) &:= \arg\min\limits_{g \in {\cal G}} \left\{ 
\langle - \tilde x, g \rangle  + J_f(g)+ \tau D_f(g^0, g)\right\},  \label{d_prox_mapping_f}
\end{align}
for any $g^0, g \in {\cal G}$. Again, $D_f$ may not be uniquely defined
since $J_f$ is not necessarily differentiable. Instead of choosing $J'_f \in \partial J_f$ similarly to
$\w'$, we can explicitly specify such selections as will be discussed later
in this paper.

The following simple result shows that
the computation of the dual prox-mapping associated with $D_f$ is equivalent
to the computation of $\nabla f$.
\begin{lemma} \label{lemma_dual_move}
Let $\tilde x \in X$ and $g^0 \in {\cal G}$ be given and $D_f(g^0, g)$
be defined in \eqnok{def_Df}. For any $\tau > 0$, let us
denote $z = [\tilde x + \tau J_f'(g^0)] / (1 + \tau)$.
Then we have $\nabla f(z) = \Pr_{{\cal G}} ( -\tilde x, g^0, \tau)$.
\end{lemma}

\begin{proof}
In view of the definition of $D_f$ in \eqnok{def_Df}, we have
\begin{align*}
\Pr_{{\cal G}} (-\tilde x,g^0, \tau) &= \arg\min\limits_{g \in {\cal G}} \left\{-\langle \tilde x
 + \tau J_f'(g^0), g \rangle + (1 + \tau) J_f(g)\right\} 
= \arg\max\limits _{g \in {\cal G}} \left\{ \langle z, g \rangle - J_f(g) \right\} 
= \nabla f(z).
\end{align*}
\end{proof}
 



\subsection{Primal-dual gradient method, Nesterov's method, and a game interpretation}
By the definition of $J_f$ in \eqnok{f_conjugate_f}, problem \eqnok{cp} is equivalent to:
\beq \label{cpsaddle_f}
\Psi^* := \min_{x \in X} \left\{h(x) + \mu \, \w(x) + \max_{g \in {\cal G} }  \langle x, g \rangle - J_f(g)\right\}.
\eeq
The primal-dual gradient method in Algorithm~\ref{algPDG} can be viewed 
as a game iteratively performed by a primal player (buyer) and a dual player (supplier)
for finding the optimal solution (order quantity and product price) of the saddle point problem in \eqnok{cpsaddle_f}. 
In this game, both the buyer and supplier have access to
their local cost $h(x)+\mu \w(x)$ and $J_f(g)$,  respectively,
as well as their interactive cost (or revenue) represented by
a bilinear function $\langle x, g \rangle$.  
Our goal is to design an algorithm such that
the buyer and supplier can achieve a equilibrium as soon as possible. 
In the proposed algorithm, the supplier first applies \eqnok{def_txt}
to predict the demand $\tilde x^t$ based on historical information, i.e., $x^{t-1}$ and $x^{t-2}$.
She then determines in \eqnok{def_yt}  the price $g^t$ in a way to
maximize the predicted profit $\langle \tilde x^t, g \rangle - J_f(g)$, regularized by
the dual prox-function $D_f(g^{t-1}, g)$ with a certain weight $\tau_t \ge 0$.
Once after the supplier has made her decision,
the buyer then determines his action
according to \eqnok{def_xt} in order to minimize the cost
$h(x) + \mu \w(x) + \langle x, g\rangle$, regularized by the primal prox-function $P(x^{t-1}, x)$ with a certain weight $\eta_t \ge 0$.

\begin{algorithm}
    \caption{The primal-dual gradient method}
    \label{algPDG}
    \begin{algorithmic}
\State Let $x^0=x^{-1} \in X$, and the nonnegative parameters $\{\tau_t\},$ $\{\eta_t\},$ and
$\{\alpha_t\}$ be given. 
\State Set $g^{0} = \nabla f(x^0)$. 

\For {$t=1, \ldots,k$}

\State Update $(x^t, g^t)$ according to
\begin{align}
\tilde x^{t}&=\alpha_t (x^{t-1}-x^{t-2})+ x^{t-1}. \label{def_txt}\\
g^{t} &= \Pr_{{\cal G}} (-\tilde x^t, g^{t-1}, \tau_t). \label{def_yt}\\ 
x^{t} &= \Pr_{X} (g^t, x^{t-1}, \eta_t). \label{def_xt}
 \end{align}

\EndFor

    \end{algorithmic}
\end{algorithm}

In order to implement the above primal-dual gradient method, it 
is more convenient to rewrite step~\eqnok{def_yt} in a form involving
the computation of gradient rather than the dual prox-mapping $\Pr_{{\cal G }}$. 
In order to do so, we shall specify explicitly the selection of 
the subgradient $J_f'$ in \eqnok{def_yt}.  
Denoting $\underline x^0 = x^0$,
we can easily see from $g^0 = \nabla f(x^0)$ that $\underline x^0 \in \partial J_f(g^0)$.
Using this relation and letting $J_f'(g^{t-1}) = \underline x^{t-1}$ in $D_f(g^{t-1}, g)$ (see \eqnok{def_Df}),
we then conclude from Lemma~\ref{lemma_dual_move} that for any $t \ge 1$,
\eqnok{def_yt} reduces to
\begin{align*} 
\underline x^t = (\tilde x^t + \tau_t \underline x^{t-1}) / (1 + \tau_t) \ \ \ \mbox{and} \ \ \
g^{t} = \nabla f(\underline x^t). 
\end{align*}
With the above selection of the dual prox-function, we can specialize the primal-dual gradient method
as follows.

\begin{algorithm} [H]
	\caption{A particular implementation of the primal-dual gradient method}
	\label{algPD}
	\begin{algorithmic}
\State 
\noindent {\bf Input:}  Let $x^0=x^{-1} \in X$, and the nonnegative parameters $\{\tau_t\},$ $\{\eta_t\},$ and
$\{\alpha_t\}$ be given. 
\State Set $\underline x^0 = x^0$.
\State {\bf for} $t = 1, 2, \ldots, k$ {\bf do}
\begin{align}
\tilde x^{t} &= \alpha_t (x^{t-1} - x^{t-2}) + x^{t-1}.\label{def_xkbar}\\
\underline x^{t} &= \left( \tilde x^t + \tau_t \underline x^{t-1} \right) / (1+\tau_t). \label{def_yt_al1} \\
g^{t} &= \nabla f(\underline x^{t}). \label{def_yt_al2} \\
x^{t} &= \Pr_{X} (g^t, x^{t-1}, \eta_t). \label{def_xk} 
\end{align}
\State {\bf end for}
	\end{algorithmic}
\end{algorithm}

Observe that
one potential problem associated with this scheme is that
the search points 
$\underline x^t$ defined in \eqnok{def_xkbar} and \eqnok{def_yt_al1}, respectively, may fall outside $X$.
As a result, we need to assume
$f$ to be differentiable over $\bbr^n$. However, it can be shown that by properly specifying $\alpha_t$ and $\tau_t$, 
we can guarantee $\underline x^t \in X$ and thus relax such restrictions on the differentiability of
$f$ (see \eqnok{def_underline1} and \eqnok{def_underline2} below). 

The above PDG method 
is related to the well-known
Nesterov's accelerated gradient (AG) method. Let us focus on a simple variant of the AG method
that has been extensively studied in the literature (e.g., \cite{Nest04,tseng08-1,Lan10-3,GhaLan12-2a,GhaLan13-1,GhaLan13-2}).
Given $(x^{t-1}, \bar x^{t-1}) \in X \times X$,
this AG algorithm updates $(x^t, \bar x^t)$ by
\begin{align}
\underline x^t &= (1-\lambda_t) \bar x^{t-1} + \lambda_t x^{t-1}, \label{AG1} \\
x^t &= \Pr_X(g^t, x^{t-1}, \eta_t), \label{AG2} \\
\bar x^t &= (1-\lambda_t) \bar x^{t-1} + \lambda_t x^t, \label{AG3}
\end{align}
for some $\lambda_t \in [0,1]$. By \eqnok{AG1} and \eqnok{AG3},
we have
\begin{align*}
\underline x^t &= (1-\lambda_t) [(1-\lambda_{t-1}) \bar x^{t-2} + \lambda_{t-1} x^{t-1}] + \lambda_t x^{t-1}\\
&= (1 - \lambda_t) [\underline x^{t-1} - \lambda_{t-1} x^{t-2} + \lambda_{t-1} x^{t-1} ]+ \lambda_t x^{t-1}\\
&= (1-\lambda_t) \underline x^{t-1} + (1-\lambda_t) \lambda_{t-1} (x^{t-1} - x^{t-2}) + \lambda_t x^{t-1}.
\end{align*}
Therefore, \eqnok{AG1} is equivalent to \eqnok{def_xkbar} and \eqnok{def_yt_al1}
with $\tau_t = (1 - \lambda_t) / \lambda_t$ and $\alpha_t = \lambda_{t-1} (1-\lambda_t) / \lambda_t$. Moreover,
\eqnok{AG2} is identical to \eqnok{def_xk}(and \eqnok{def_xt}), and \eqnok{AG3} basically defines
the output of the AG algorithm as an ergodic mean of the iterates $x^t$. We then conclude that
the above variant of Nesterov's AG method is a special case of Algorithm~\ref{algPD} (and Algorithm~\ref{algPDG}). 
It should be noted, however, that Algorithm~\ref{algPDG} provides more flexibility in
the specification of parameters, which will be used later in the development of
the RPDG method. Moreover,
the presentation of the PDG method helps us to reveal a natural game interpretation
out of the intertwined and somehow mysterious updating of the three search sequences in the
AG method. 

Algorithm~\ref{algPDG} is also closely related to 
Chambolle and Pock's primal-dual method for solving
saddle point problems \cite{ChamPoc11-1}, which explains the origin of its name. 
Two versions of primal-dual methods 
were discussed in \cite{ChamPoc11-1}.
One is designed for
solving general saddle point problems without assuming the strong convexity of $J_f$
and the other one is to deal with 
the case when $J_f$ is strongly convex
by incorporating an additional extrapolation step. As pointed out in Remark 3 of \cite{ChamPoc11-1},
the rate of convergence for the latter primal-dual method is only suboptimal
for solving \eqnok{cp}
as it uses a weaker termination criterion. 
On the other hand, the PDG method does not involve any additional extrapolation steps
and so it shares a similar
scheme to the basic version of the primal-dual method in \cite{ChamPoc11-1}. 
Moreover, the original primal-dual methods in \cite{ChamPoc11-1} do not employ general prox-functions, 
which, as shown in Lemma~\ref{lemma_dual_move}, is crucial to relate the dual 
step \eqnok{def_yt} to the computation of the gradients. It should be noted that some recent extensions of
the primal-dual method in \cite{CheLanOu13-1,DangLan14-1,ChamPoc14-1}
indeed consider the incorporation of prox-functions, but
restricted to problems without strong convexity. Hence, none of 
these earlier primal-dual methods
can be viewed as a generalized accelerated gradient method.

\subsection{Convergence properties of the primal-dual gradient method} \label{sub_PDG_conv}
Our goal in this subsection is to show that Algorithm~\ref{algPDG}
exhibits an optimal rate of convergence for solving problem \eqnok{cp}. It is worth mentioning that
our analysis significantly differs from the previous studies on optimal gradient methods
and those on primal-dual methods for saddle point problems.

Given a pair of feasible solutions $\bar z = (\bar x,  \bar g)$ and $z = (x, g)$ of \eqnok{cpsaddle_f},
we define the primal-dual gap function $Q_f(\bar z, z)$ by
\beq \label{def_gap_f}
Q_f(\bar z, z) := \left[h(\bar x) + \mu \w(\bar x) + \langle \bar x, g \rangle -  J_f(g)\right]
- \left[\cX(x) + \mu \w(x) + \langle x, \bar g \rangle - J_f(\bar g) \right].
\eeq
It can be easily seen that $\bar z$ (resp., $z$) is
an optimal solution of \eqnok{cpsaddle_f} if and only if $Q_f(\bar z, z) \le 0$
for any $z \in X \times {\cal G}$ (resp., $Q_f(\bar z, z) \ge 0$ for any $\bar z \in X \times {\cal G}$).  
Therefore, one can assess the solution quality  of $\bar z$ by
the primal-dual optimality gap:
\beq \label{def_pd_gap}
\gap(\bar z) := \max_{z \in X \times {\cal G} } Q_f(\bar z, z). 
\eeq
It should be noted that $\gap(\bar z)$ may not be well-defined, for example, when
$X$ is unbounded and $h$ is not strictly convex.
In these cases, we can define a slightly modified primal-dual gap 
\beq \label{def_pd_gap_unbounded}
\gap^*(\bar z) := \max \left\{Q_f(\bar z, z): x = x^*, g \in {\cal G} \right\}
\eeq
for an arbitrary optimal solution $x^*$ of \eqnok{cp}.
Since $J_f$ is strongly convex, $\gap^*$ is well-defined. 

The following result establishes some relationship between 
the primal optimality gap $\Psi(\bar z) - \Psi^*$ and the above primal-dual
optimality gaps.

\begin{lemma}
Let $\bar z = (\bar x, \bar g) \in X\times \mathcal{G}$ be a given pair of feasible solutions of \eqnok{cpsaddle_f} and
denote $\bar g^* = \nabla f(\bar x)$. Also let $z^*=(x^*, g^*)$
be a pair of optimal solutions of \eqnok{cpsaddle_f}. Then we have
\beq \label{PDgap1}
\Psi(\bar x) - \Psi(x^*) = Q_f((\bar x, g^*),(x^*, \bar g^*)) \le \gap^*(\bar z).
\eeq
If in addition, $X$ is bounded, then
\beq \label{PDgap2}
\gap^*(\bar z) \le \gap(\bar z).
\eeq
\end{lemma}

\begin{proof}
It follows from the definitions of $\bar g^*$, $\gap^*$ and the gap function $Q_f$ that
\begin{align*}
\Psi(\bar x) - \Psi(x^*) 
&= Q_f((\bar x, g^*),(x^*, \bar g^*))\\
&= 
 [ h(\bar x) + \mu \w(\bar x) + \max_{g\in \mathcal{G}} \langle \bar x, g \rangle - J_f(g)]
- [h(x^*) + \mu \w(x^*) + \langle x^*, g^* \rangle - J_f(g^*) ] \nn \\
&\le [h(\bar x) +  \mu \w(\bar x) + \max_{g\in \mathcal{G}}  \langle \bar x, g \rangle - J_f(g)]
- [ h(x^*) + \mu \w(x^*) + \langle x^*, \bar g \rangle - J_f(\bar g) ] \nn\\
&= \gap^*(\bar z). 
\end{align*}
Relation \eqnok{PDgap2} follows directly from the definitions of $\gap^*$ and $\gap$.
\end{proof}

\vgap

Theorem~\ref{cor_det1} below 
describes the main convergence properties of the PDG method. 
More specifically, we provide in Theorem~\ref{cor_det1}.a) a constant stepsize policy 
which works for the strongly convex case where $\mu > 0$,
and a different parameter setting that works for the non-strongly convex case with $\mu = 0$ in
Theorem~\ref{cor_det1}.b). Note that for the strongly convex case, we estimate the solution quality for
the iterates $x^t, t = 1, \ldots, k$, as well as that for their ergodic mean 
\beq \label{def_avg_output_s}
\bar x^k = (\tsum_{t=1}^k \theta_t)^{-1} \tsum_{t=1}^k (\theta_t x^t)
\eeq
for some $\theta_t \ge 0$, while only establishing the error bounds for $\bar x^k$ for
the non-strongly convex case. 
We put the proof of Theorem~\ref{cor_det1} in Section~\ref{sec_analysis}
since it shares many basic elements with the convergence analysis of the RPDG method.

\begin{theorem} \label{cor_det1}
Let $x^*$ be an optimal solution of \eqnok{cp}, $x^k$ and $\bar x^k$ be defined in \eqnok{def_xt}
and \eqnok{def_avg_output_s}, respectively.
\begin{itemize}
\item [a)] Suppose that  $\mu > 0$ and that $\{\tau_t\}$, $\{\eta_t\}$, $\{\alpha_t\}$ and $\{\theta_t\}$ are set to
\beq \label{constant_step_PDG1}
\tau_t = \sqrt{\tfrac{2 L_f}{\mu}}, \ \ \ \eta_t = \sqrt{2 L_f\mu}, \ \ \ \alpha_t = \alpha \equiv
\tfrac{ \sqrt{2 L_f/\mu}}{1 + \sqrt{2 L_f/\mu}}, \ \ \ \mbox{and} \ \ \ \theta_t = \tfrac{1}{\alpha^t}, \ \forall t = 1, \ldots, k.
\eeq 
Then,
\begin{align}
P(x^k,x^*) &\le \tfrac{\mu+ L_f }{\mu} \alpha^k  P(x^0, x^*), \label{PDG_cor1_result}\\
\Psi(\bar x^k) - \Psi(x^*) &\le \gap^*(\bar z^k) \le  \mu (1-\alpha)^{-1}  \left[  1  + \tfrac{L_f}{\mu}(2+\tfrac{L_f}{\mu})\right]  \alpha^k P(x^0,x^*),  \label{PDG_cor1_result1}\\
\Psi(\bar x^k) - \Psi(x^*) &\le \gap(\bar z^k)  \le \mu (1-\alpha)^{-1} \left[  1  + \tfrac{ L_f}{\mu}(2+\tfrac{L_f}{\mu})\right] \alpha^k \max_{x \in X} P(x^0, x). \label{PDG_cor1_result2}
\end{align}
\item [b)] Suppose that $\{\tau_t\}$, $\{\eta_t\}$, $\{\alpha_t\}$ and $\{\theta_t\}$ are set to
\beq \label{var_step_PDG1}
\tau_t =\frac{t-1}{2}, \ \ \ \eta_t = \frac{4 L_f}{t},  \ \ \ \alpha_t = \frac{t-1}{t} \ \ \ \mbox{and} \ \ \  \theta_t = t, \ \forall t = 1, \ldots,k.
\eeq
Then,
\begin{align}
\Psi(\bar x^k) - \Psi(x^*) &\le \gap^*(\bar z^k) \le \frac{8 L_f }{  k (k+1)} P(x^0, x^*),  \label{PDG_cor2_result}\\
\Psi(\bar x^k) - \Psi(x^*) &\le \gap(\bar z^k)  \le \frac{8 L_f }{k (k+1)} \max_{x \in X} P(x^0, x).\label{PDG_cor2_result1}
\end{align}
\end{itemize}
\end{theorem}

\vgap

Observe that when the algorithmic parameters are set to \eqnok{constant_step_PDG1}, by using an inductive argument,
we can easily show that
\beq \label{def_underline1}
\underline x^k = (1-\alpha^2) x^{k-1} + (1-\alpha) \tsum_{t=1}^{k-2} (\alpha^{k-t} x^t) + \alpha^k x^0.
\eeq
In other words,  $\underline x^k$ can be written as a convex combination of $x^0, \ldots, x^{k-1}$ 
and hence $\underline x^k \in X$ for any $k \ge 1$.
Similarly, when the algorithmic parameters are set to \eqnok{var_step_PDG1}, 
we can show by using induction that
\beq \label{def_underline2}
\underline x^k = \tfrac{2 (2k-1)}{k (k+1)} x^{k-1} + \tfrac{2}{k(k+1)} \tsum_{t=1}^{k-2} (i x^i),
\eeq
which implies $\underline x^k \in X$. 
Therefore, we only need to assume the differentiability of $f$ over $X$ rather than
the whole $\bbr^n$.

In view of the results obtained in Theorem~\ref{cor_det1},
the primal-dual gradient method is an optimal method for convex optimization.
In fact, the rates of convergence in \eqnok{PDG_cor1_result1}, \eqnok{PDG_cor1_result2},
\eqnok{PDG_cor2_result} and \eqnok{PDG_cor2_result1} associated with the ergodic mean
$\bar z^k$ have employed
the primal-dual optimality gaps $g^*(\bar z^k)$ and $g(\bar z^k)$, which are stronger than
the primal optimality gap
$\Psi(\bar x^k) - \Psi(x^*)$ used in the previous studies for accelerated gradient methods. 
Moreover, whenever $X$ is bounded, the primal-dual optimality gap $g(\bar z^k)$ gives us a computable online
accuracy certificates to check the quality of the solution $\bar z^k$ (see \cite{lns11,GhaLan12-2a} for some related discussions).
Also observe that each iteration of the PDG method requires the computation of $\nabla f$, and hence all the $m$ components
$\nabla f_i$. In the next section, we will develop a randomized PDG method that can possibly save the
number of gradient evaluations for $\nabla f_i$ by utilizing the finite-sum structure of problem \eqnok{cp}.

\setcounter{equation}{0}
\section{Randomized primal-dual gradient methods} \label{sec_ORIG}
In this section, we present a randomized primal-dual gradient (RPDG) method
which needs to compute the gradient of only one randomly selected component function 
$f_i$ at each iteration. We show that RPDG can possibly achieve a better complexity than PDG
in terms of the total number of gradient evaluations. 
\subsection{Multi-dual-player reformulation and the RPDG algorithm}

We start by introducing a different saddle point reformulation of \eqnok{cp} than \eqnok{cpsaddle_f}.
Let $J_i: {\cal Y}_i \to \bbr$ be the conjugate functions of $f_i$ and
${\cal Y}_i$, $i = 1, \ldots, m$, denote the dual spaces where the
gradients of $f_i$ reside. For the sake of notational convenience, let us denote
$J(y) := \tsum_{i=1}^m J_i(y_i)$,
${\cal Y} := {\cal Y}_1 \times {\cal Y}_2 \times \ldots \times {\cal Y}_m$,
and $y = (y_1; y_2; \ldots; y_m)$ for any $y_i \in {\cal Y}_i$, $i = 1, \ldots, m$.
Clearly, we can reformulate problem \eqnok{cp}
equivalently as a saddle point problem:
\beq \label{cpsaddle}
\Psi^* := \min_{x \in X} \left\{h(x) + \mu \, \w(x) + \max_{y \in {\cal Y} }  \langle x, U y \rangle - J(y)\right\},
\eeq
where
$U \in \bbr^{n \times nm}$ is given by
\beq \label{def_U}
U := \left[I, I, \ldots, I\right].
\eeq
Here $I$ is the identity matrix in $\bbr^n$.
Given a pair of feasible solutions $\bar z = (\bar x, \bar y)$ and $z = (x, y)$ of \eqnok{cpsaddle},
we define the primal-dual gap function $Q(\bar z, z)$ by
\beq \label{def_gap}
Q(\bar z, z) := \left[h(\bar x) + \mu \w(\bar x) + \langle \bar x, U y \rangle -  J(y)\right]
- \left[\cX(x) + \mu \w(x) + \langle x, U \bar y \rangle - J(\bar y) \right].
\eeq
It is well-known that $\bar z \in Z \equiv X \times {\cal Y}$ is
an optimal solution of \eqnok{cpsaddle} if and only if $Q(\bar z, z) \le 0$
for any $z \in Z$. 

Since $J_i, i = 1, \ldots, m$, are strongly convex
with modulus $\sigma_i = 1/ L_i$ w.r.t. $\|\cdot\|_*$, we 
can define their associated dual prox-functions and dual prox-mappings as
\begin{align} 
D_i(y^0_i, y_i) &:= J_i(y_i) - [J_i(y^0_i) + \langle J_i'(y^0_i), y_i - y^0_i \rangle], \label{def_dual_prox}\\
\Pr_{{\cal Y}_i} (-\tilde x, y_i^0, \tau) &:= \arg\min\limits_{y_i \in {\cal Y}_i} \left\{ 
\langle - \tilde x, y \rangle  + J_i(y_i)+ \tau D_i(y^0_i, y_i)\right\},  \label{d_prox_mapping}
\end{align}
for any $y^0_i, y_i \in {\cal Y}_i$. 
Accordingly, we define
\beq \label{def_D}
D(\tilde y, y) := \tsum_{i=1}^m D_i(\tilde y_i, y_i).
\eeq
Again, $D_i$ may not be uniquely defined
since $J_i$ are not necessarily differentiable. However, we will discuss how to
specify the particular selection of $J'_i \in \partial J_i$ later
in this subsection.

\vgap

We are now ready to describe the randomized primal-dual method, which
is obtained by properly modifying the primal-dual gradient method as follows. Firstly, in \eqnok{def_yt_r}, 
we only compute a randomly selected
dual prox-mapping $\Pr_{{\cal Y}_{i}}$ rather than the dual prox-mapping $\Pr_{{\cal G}}$ as in Algorithm~\ref{algPDG}.
Secondly, in addition to the primal prediction step~\eqnok{def_txt_r}, we add a new dual prediction step \eqnok{def_tyt_r}, and 
then use the predicted 
dual variable $\tilde y^t$ for the computation of the new search point $x^t$
in \eqnok{def_xt_r}. It can be easily seen that
the RPDG method reduces to the PDG method
whenever this algorithm is directly applied to \eqnok{cpsaddle_f} (i.e., $m = 1$, ${\cal Y}_1 = {\cal G}$, and
$J_1 = J_f$) .

\begin{algorithm}
    \caption{A randomized primal-dual gradient (RPDG) method}
    \label{algRPDG}
    \begin{algorithmic}
\State Let $x^0=x^{-1} \in X$, and the nonnegative parameters $\{\tau_t\},$ $\{\eta_t\},$ and 
$\{\alpha_t\}$ be given. 
\State Set $y^0_i = \nabla f_i(x^0)$, $i=1, \ldots, m$.

\For {$t=1, \ldots,k$}
\State Choose $i_t$ according to $\prob\{i_t = i\} = p_i$, $i=1, \ldots, m$.
\State Update $z^t=(x^t, y^t)$ according to
\begin{align}
\tilde x^{t}&=\alpha_t (x^{t-1}-x^{t-2})+ x^{t-1}. \label{def_txt_r}\\
y_i^{t} &= 
\begin{cases}
\Pr_{{\cal Y}_{i}} (-\tilde x^t, y_i^{t-1}, \tau_t),& i = i_t, \\
 y_i^{t-1}, & i \neq i_t.
\end{cases} \label{def_yt_r}\\
\tilde y^t_i &= 
\begin{cases}
p_i^{-1} (y^t_i - y^{t-1}_i) + y^{t-1}_i, & i = i_t, \\
y_i^{t-1}, & i\neq i_t.  
\end{cases}. \label{def_tyt_r} \\
x^{t} &= \Pr_{X} (\tsum_{i=1}^m \tilde y_i^t, x^{t-1}, \eta_t). \label{def_xt_r}
 \end{align}

\EndFor
    \end{algorithmic}
\end{algorithm}

Similarly to the PDG method,
the RPDG method can be viewed 
as a game iteratively performed by a buyer and $m$ suppliers
for finding the solutions (order quantities and product prices) of the saddle point problem in \eqnok{cpsaddle}. 
In this game, both the buyer and suppliers have access to
their local cost $h(x)+\mu \w(x)$ and $J_i(y_i)$,  respectively,
as well as their interactive cost (or revenue) represented by
a bilinear function $\langle x, y_i \rangle$.  
Also, the buyer has to purchase the same amount of products
from each supplier (e.g., for fairness). 
Although there are $m$ suppliers,
in each iteration only a randomly chosen supplier can make price changes according to \eqnok{def_yt_r}
using the predicted demand $\tilde x^t$. In order to understand the buyer's decision in \eqnok{def_xt_r}, let us first denote
\beq  \label{def_hyt_r}
\hat y_i^t := \Pr_{{\cal Y}_{i}} (-\tilde x^t, y_i^{t-1}, \tau_t), \ \ i = 1, \ldots, m; \, t = 1, \ldots, k.
\eeq
In other words, $\hat y_i^t$, $i = 1, \ldots, m$, denote the prices that all the suppliers can possibly
set up at iteration $t$. 
Then we can see that
\beq \label{exp_tyt}
\bbe_t[\tilde y^t_i] = \hat y^t_i.
\eeq
Indeed, we have 
\beq \label{def_tyt_alt_r1}
y_i^t = \begin{cases}
\hat y_i^t,& i = i_t, \\
 y_i^{t-1}, & i \neq i_t.
\end{cases} 
\eeq
Hence $\bbe_t[y_i^t] = p_i \hat y_i^t + (1 - p_i) y_i^{t-1}$, $i = 1, \ldots, m$.
Using this identity in the definition of $\tilde y^t$ in \eqnok{def_tyt_r}, we
obtain \eqnok{exp_tyt}. Instead of using $\tsum_{i=1}^m \hat y^t_i$ in determining
his order in \eqnok{def_xt_r},
the buyer notices that only one supplier has made a change on the price, and thus 
uses $\tsum_{i=1}^m \tilde y^t_i$ to predict the case when all the dual players
would modify the prices simultaneously. 

In order to implement the above RPDG method, we shall explicitly specify 
the selection of the subgradient $J'_{i_t}$ in the definition of
the dual prox-mapping in \eqnok{def_yt_r}. 
Denoting $\underline x^0_i = x^0$, $i = 1, \ldots, m$,
we can easily see from $y^0_i = \nabla f_i(x^0)$ that $\underline x^0_i \in \partial f_i^*(y_i^0)$, $i = 1, \ldots, m$.
Using this relation and letting $J_i'(y_i^{t-1}) = \underline x^{t-1}_i$ in the definition of $D_i(y_i^{t-1}, y_i)$
in \eqnok{def_yt_r} (see \eqnok{def_dual_prox}),
we then conclude from Lemma~\ref{lemma_dual_move} (with $J_f = J_{i_t}$ and $D_f = D_{i_t}$) and \eqnok{def_yt_r} that for any $t \ge 1$,
\begin{align*} 
\underline x^t_{i_t} &= (\tilde x^t + \tau_t \underline x^{t-1}_{i_t}) / (1 + \tau_t),  \ \ \  \underline x^t_i = \underline x^{t-1}_i, \ \forall i \neq i_t;\\  
y_{i_t}^{t} &= \nabla f_{i_t}(\underline x^t_{i_t}), \ \ \ y_i^t = y_i^{t-1}, \ \forall i \neq i_t.
\end{align*}
Moreover, observe that the computation of $x^t$ in \eqnok{def_xt_r}
requires an involved computation of $\tsum_{i=1}^m \tilde y_i^t$. In order to save computational time,
we suggest to compute this quantity in a recursive manner as follows. Let us denote $g^t \equiv \tsum_{i=1}^m y_i^t$.
Clearly, in view of the fact that $y^t_i = y^{t-1}_i$, $\forall i \neq i_t$, we have
\[
g^t = g^{t-1} + (y^t_{i_t} - y^{t-1}_{i_t}).
\]
Also, by the definition of $g^t$ and \eqnok{def_tyt_r}, we have
\begin{align*}
\tsum_{i=1}^m \tilde y^t_i 
&= \tsum_{i \neq i_t} y_i^{t-1} + p_{i_t}^{-1} (y^t_{i_t} - y^{t-1}_{i_t}) + y^{t-1}_{i_t} \\
&= \tsum_{i=1}^m y_i^{t-1} +p_{i_t}^{-1} (y^t_{i_t} - y^{t-1}_{i_t}) \\
&=g^{t-1} + p_{i_t}^{-1} (y_{i_t}^t - y_{i_t}^{t-1}).
\end{align*}
Incorporating these two ideas mentioned above, we present an efficient implementation of the
RPDG method in Algorithm~\ref{algRGB}.

\begin{algorithm}
    \caption{An efficient implementation of the RPDG method}
    \label{algRGB}
    \begin{algorithmic}
\State Let $x^0=x^{-1} \in X$, and nonnegative parameters $\{\alpha_t\}$, $\{\tau_t\},$ 
and $\{\eta_t\}$ be given. 
\State Set $\underline x_i^0 = x^0$, $y^0_i = \nabla f_i( x^0)$, $i=1, \ldots, m$, and $g^0 = \tsum_{i=1}^m y^0_i$.

\For {$t=1, \ldots,k$}
\State Choose $i_t$ according to $\prob\{i_t = i\} = p_i$, $i=1, \ldots, m$.
\State Update $z^t := (x^t, y^t)$ by
\begin{align}
\tilde x^{t}&=\alpha_t (x^{t-1}-x^{t-2})+ x^{t-1}. \label{def_txt_r1}\\
\underline x_i^t &= 
\begin{cases}
 (1+ \tau_t)^{-1} \left(\tilde x^t + \tau_t \underline x_i^{t-1}\right), & i = i_t,\\
\underline x_i^{t-1}, & i \neq i_t.
\end{cases} \label{def_uxt_r}\\ 
y_i^{t} &= 
\begin{cases}
\nabla f_i(\underline x_i^t),& i = i_t, \\
 y_i^{t-1}, & i \neq i_t.
\end{cases} \label{def_yt_r1}\\
x^{t} &= \Pr_{X} (g^{t-1} + p_{i_t}^{-1} (y_{i_t}^t - y_{i_t}^{t-1}), x^{t-1}, \eta_t). \label{def_xt_r1}\\
g^t &= g^{t-1} + y_{i_t}^t - y_{i_t}^{t-1}.\label{def_gt}
\end{align}
\EndFor

    \end{algorithmic}
\end{algorithm}

Clearly, the RPDG method is an incremental gradient type method since each iteration of this algorithm
involves the computation of the gradient $\nabla f_{i_t}$ of only one component function.
As shown in the following Subsection, such an randomization scheme can lead to
significantly savings on the total number of gradient evaluations, at the expense of
more primal prox-mappings.

It should also be noted that due to the randomness in the RPDG method, we can not guarantee that
$\underline x_i^t \in X$ for all $i =1, \ldots, m$, and $t \ge 1$ in general, even though we do have all the iterates $x^t \in X$.
That is why we need to make the assumption that $f_i$'s are differentiable over $\bbr^n$
for the RPDG method.

\subsection{The convergence of the RPDG algorithm}
Our goal in this subsection is to describe the convergence 
properties of the RPDG method for the strongly convex case when $\mu > 0$.
Generalization of the RPDG method for the non-strongly convex case will be discussed
in Section~\ref{sec_general}. 

Theorem~\ref{theorem_RPDG_s} below states some general convergence properties of
RPDG. Similar to PDG method, we provide bounds on $\bbe[P(x^k, x^*)]$ and $\bbe[\Psi(\bar x^k) - \Psi(x^*)]$.
However, we cannot provide a bound on the expected primal-dual gap $\bbe[\gap(\bar x^k)]$ 
even though our analysis
for the RPDG algorithm still relies on the primal-dual gap function $Q$ in \eqnok{def_gap}
(see \cite{DangLan14-1} for some relevant disucssions).

\begin{theorem} \label{theorem_RPDG_s}
Suppose that $\{\tau_t\}$, $\{\eta_t\}$, and $\{\alpha_t\}$ in the RPDG method are
set to \beq \label{constant_step_RPDG}
\tau_t = \tau, \ \ \ \eta_t = \eta, \ \ \ \mbox{and} \ \ \ \alpha_t = \alpha, 
\eeq 
for any $t \ge 1$  such that
\begin{align}
(1-\alpha) (1+\tau) &\le p_i, i=1,\ldots, m, \label{cond_s1}\\
 \eta &\le \alpha (\mu + \eta), \label{cond_s2}\\
\eta \tau p_i  &\ge 4 L_i , i = 1, \ldots, m, \label{cond_s3} 
\end{align}
for some $\alpha \in (0,1)$.
Then, for any $k \ge 1$, we have
\begin{align} 
\bbe[P(x^k, x^*)] 
&\le  \left (1 + \tfrac{L_f \alpha}{(1-\alpha)\eta} \right) \alpha^k P(x^0, x^*),\label{RPDG_s_main}\\
\bbe[\Psi(\bar x^k) - \Psi(x^*)] 
&\le \alpha^{k/2}\left(\alpha^{-1}\eta + \tfrac{3-2\alpha}{1-\alpha}L_f+\tfrac{2 L_f^2 \alpha}{(1-\alpha)\eta}\right)P(x^0,x^*), \label{RPDG_s_main1}
\end{align}
where $\bar{x}^k=(\tsum_{t=1}^{k}\theta_t)^{-1}\tsum_{t=1}^{k}(\theta_tx^t)$ with $\{\theta_t\}$ defined as in \eqnok{constant_step_PDG1}, and $x^*$ denotes the optimal solution
of problem \eqnok{cp}, and the expectation is taken w.r.t. $i_1, \ldots, i_k$.
\end{theorem}

\vgap

We now provide a few specific selections of $p_i$, $\tau$, $\eta$, and $\alpha$ 
satisfying \eqnok{cond_s1}-\eqnok{cond_s3} and establish
the complexity of the RPDG method for computing
a stochastic $\epsilon$-solution of problem \eqnok{cp}, i.e., a point $\bar x \in X$ s.t. $\bbe[P(\bar x, x^*)] \le \epsilon$,
as well as a stochastic $(\epsilon, \lambda)$-solution of 
problem \eqnok{cp}, i.e., a point $\bar x \in X$ s.t.
$\Prob\{P(\bar x, x^*) \le \epsilon\} \ge 1 - \lambda$ for some $\lambda \in (0,1)$. 
Moreover, in view of \eqnok{RPDG_s_main1},
similar complexity bounds of the RPDG method can 
be established in terms of the primal optimality gap, i.e. $\bbe[\Psi(\bar x)-\Psi^*]$.

The following corollary shows the convergence of RPDG under a non-uniform
distribution for the random variables $i_t$, $t=1, \ldots,k$.

\begin{corollary} \label{cor_RPDG1}
Suppose that $\{i_t\}$ in the RPDG method are distributed over $\{1, \ldots, m \}$ 
according to
\beq \label{nonuniform}
p_i = \Prob\{i_t = i\} = \tfrac{1}{2m} + \tfrac{L_i}{2 L},  i = 1, \ldots,m.
\eeq
Also assume that $\{\tau_t\}$, $\{\eta_t\}$, and $\{\alpha_t\}$ are set to \eqnok{constant_step_RPDG} with
\beq \label{nonuniform_stepsize}
\tau = \tfrac{\sqrt{(m-1)^2 + 4 m C} - (m-1)}{2m}, \ \ 
\eta = \tfrac{\mu \sqrt{(m-1)^2 + 4m  C} + \mu (m-1)}{2}, \ \ \mbox{and} \ \
\alpha= 1 - \tfrac{1}{(m+1) + \sqrt{(m-1)^2 + 4 m C}},
\eeq
where
\beq \label{def_C}
C = \tfrac{8 L}{\mu}.
\eeq
Then for any $k \ge 1$, we have
\begin{align}
\bbe[P(x^k, x^*)] &\le(1+\tfrac{3L_f}{\mu})  \alpha^k  P(x^{0}, x^*), \label{nonuniform_result}\\
\bbe[\Psi(\bar x^k) - \Psi^*] &\le \alpha^{k/2}(1-\alpha)^{-1} \left[\mu +  2L_f + \tfrac{L_f^2}{\mu}\right] P(x^{0}, x^*). \label{nonuniform_result1}
\end{align}
As a consequence, the number of iterations performed by the RPDG method to find a stochastic $\epsilon$-solution
and a stochastic $(\epsilon, \lambda)$-solution of \eqnok{cp}, in terms of the distance to the optimal solution, i.e., $\bbe[P(x^k,x^*)]$, 
can be bounded by $K(\epsilon,C)$ and $K(\lambda \epsilon, C)$, respectively, where
\begin{align}
K(\epsilon, C) & := \left[(m+1) + \sqrt{(m-1)^2 + 4 m C} \right] \log \left[(1+\tfrac{3L_f}{\mu})\tfrac{P(x^{0}, x^*)}{\epsilon}\right]. \label{non_uniform_com1}
\end{align}
Similarly, the total number of iterations performed by the RPDG method to find a stochastic $\epsilon$-solution
and a stochastic $(\epsilon, \lambda)$-solution of \eqnok{cp}, in terms of the primal optimality gap, i.e., $\bbe[\Psi(\bar x^k)-\Psi^*]$, 
can be bounded by $\tilde K(\epsilon,C)$ and $\tilde K(\lambda \epsilon, C)$, respectively, where
\begin{align}
\tilde K(\epsilon, C) & := 2\left[(m+1) + \sqrt{(m-1)^2 + 4 m C} \right] \log \left[2(\mu +2L_f+\tfrac{L_f^2}{\mu})(m+\sqrt{mC})\tfrac{P(x^{0}, x^*)}{\epsilon}\right]. \label{non_uniform_com2}
\end{align}
\end{corollary}

\begin{proof}
It follows from \eqnok{nonuniform_stepsize} that
\[
(1-\alpha)(1+\tau) = 1/(2m) \le p_i, \ \
(1-\alpha) \eta =(\alpha-1/2) \mu \le \alpha \mu , \ \ \mbox{and} \ \ \eta \tau p_i = \mu  C p_i \ge 4 L_i,
\]
and hence that the conditions in \eqnok{cond_s1}-\eqnok{cond_s3} are satisfied.
Notice that by the fact that $\alpha \ge 3/4, \ \  \forall m\ge 1$ and \eqnok{nonuniform_stepsize}, we
have
\[
1 + \tfrac{L_f \alpha}{(1-\alpha)\eta} = 1+L_f\tfrac{\alpha}{(\alpha-1/2)\mu} \le 1+\tfrac{3L_f}{\mu}.
\]
Using the above bound in \eqnok{RPDG_s_main}, we obtain \eqnok{nonuniform_result}.
It follows from the facts $(1-\alpha) \eta \le \alpha \mu$, $1/2 \le \alpha \le 1, \forall m\ge 1$,
and 
$
\eta \ge \mu \sqrt{C} > 2\mu
$
that
\[
 \alpha^{-1}\eta+\tfrac{3-2\alpha}{1-\alpha}L_f+\tfrac{2L_f^2\alpha}{(1-\alpha)\eta}
\le (1-\alpha)^{-1}(\mu+2L_f+\tfrac{L_f^2}{\mu}) .
\]
Using the above bound in \eqnok{RPDG_s_main1}, we obtain \eqnok{nonuniform_result1}.
Denoting $D \equiv (1+ \tfrac{3L_f}{\mu})P(x^{0}, x^*)$,
we conclude from \eqnok{nonuniform_result} and the fact that $\log x \le x-1$ for any $x \in (0,1)$ that
\begin{align*}
\bbe[P(x^{K(\epsilon,C)}, x^*)] 
&\le D \alpha^\frac{\log (D/\epsilon)}{1-\alpha}  
 \le D \alpha^\frac{\log (D/\epsilon)}{-\log \alpha} 
 \le D \alpha^\frac{\log (\epsilon/ D)}{\log \alpha}  = \epsilon.
\end{align*}
Moreover, by Markov's inequality, \eqnok{nonuniform_result} and the fact that $\log x \le x-1$ for any $x \in (0,1)$, we have
\begin{align*}
\Prob\{P(x^{K(\lambda \epsilon, C)}, x^*) > \epsilon \} 
&\le \tfrac{1}{\epsilon} \bbe[P(x^{K( \lambda\epsilon,C)}, x^*)]
 \le \tfrac{D}{\epsilon} \alpha^\frac{\log (D/(\lambda \epsilon))}{1-\alpha} 
 \le \tfrac{D}{\epsilon} \alpha^\frac{\log (\lambda \epsilon/D)}{\log \alpha} = \lambda.
\end{align*}
The proofs for the complexity bounds in terms of the primal optimality gap is similar and hence the details are skipped.
\end{proof}

The non-uniform distribution in \eqnok{nonuniform} requires the estimation of the Lipschitz constants $L_i$, $i = 1, \ldots, m$.
In case such information is not available, we can use a uniform distribution for $i_t$, and as a result,
the complexity bounds will depend on a larger condition number given by $m \max_{i=1, \ldots, m} L_i /\mu$.
However, if we do have $L_1 = L_2 = \cdots=L_m$, then the results obtained by using a uniform distribution is 
slightly sharper than the one by using a non-uniform distribution in Corollary~\ref{cor_RPDG1}.

\begin{corollary} \label{cor_RPDG2}
Suppose that $\{i_t\}$ in the RPDG method are uniformly distributed over $\{1, \ldots, m \}$ according to
\beq \label{uniform}
p_i = \Prob\{i_t = i\} = \tfrac{1}{m},  i = 1, \ldots,m.
\eeq
Also assume that $\{\tau_t\}$, $\{\eta_t\}$, and $\{\alpha_t\}$ are set to \eqnok{constant_step_RPDG} with
\beq \label{uniform_stepsize}
\tau = \tfrac{\sqrt{(m-1)^2 + 4 m \bar C} - (m-1)}{2m}, \ \ \eta = \tfrac{\mu \sqrt{ (m-1)^2 + 4m \bar C} + \mu (m-1)}{2}, \ \
\mbox{and} \ \ \alpha = 1 - \tfrac{2}{(m+1) + \sqrt{(m-1)^2 + 4 m \bar C}},
\eeq
where
\beq \label{def_barC}
\bar C := \tfrac{4 m}{\mu} \max\limits_{i=1,\ldots, m} L_i.
\eeq
Then we have
\begin{align} 
\bbe[P(x^k, x^*)] &\le  (1+\tfrac{L_f}{\mu})\alpha^k P(x^{0}, x^*), \label{uniform_result}\\ 
\bbe[\Psi(\bar x^k) - \Psi^* ]&\le \alpha^{k/2} (1-\alpha)^{-1} \left(\mu +  2L_f+ \tfrac{ L_f^2}{\mu} \right) P(x^{0}, x^*). \label{uniform_result1}
\end{align}
for any $k \ge 1$. As a consequence, the number of iterations performed by the RPDG method to find a stochastic $\epsilon$-solution
and a stochastic $(\epsilon, \lambda)$-solution of \eqnok{cp}, in terms of the distance to the optimal solution, i.e., $\bbe[P(x^k,x^*)]$, can be bounded by
$K_u(\epsilon, \bar C)$ and $K_u(\lambda \epsilon, \bar C)$, respectively, where 
\[
K_u(\epsilon, \bar C):=\tfrac{(m+1)+\sqrt{(m-1)^2+4m\bar C}}{2}\log \left[(1+\tfrac{L_f}{\mu})\tfrac{P(x^0,x^*)}{\epsilon}\right].
\]
Similarly, the total number of iterations performed by the RPDG method to find a stochastic $\epsilon$-solution
and a stochastic $(\epsilon, \lambda)$-solution of \eqnok{cp}, in terms of the primal optimality gap, 
i.e., $\bbe[\Psi(\bar x^k)-\Psi^*]$, can be bounded by $\tilde K(\epsilon,\bar C)/2$ and $\tilde K(\lambda \epsilon,\bar C)/2$, respectively,
where $\tilde{K}(\epsilon,\bar C)$ is defined in \eqnok{non_uniform_com2}.
\end{corollary}

\begin{proof}
It follows from \eqnok{uniform_stepsize} that
\[
(1-\alpha)(1+\tau) = 1/m = p_i, \ \
(1-\alpha) \eta -\alpha \mu = 0, \ \ \mbox{and} \ \ \eta \tau = \mu \bar C \ge 4m L_i,
\]
and hence that the conditions in \eqnok{cond_s1}-\eqnok{cond_s3} are satisfied.
By the identity $(1-\alpha) \eta =\alpha \mu$, we have
\[
1+\tfrac{L_f\alpha}{(1-\alpha)\eta} = 1+\tfrac{L_f}{\mu}.
\]
Using the above bound in \eqnok{RPDG_s_main}, we obtain \eqnok{uniform_result}. Moreover, note that 
$\eta \ge \mu \sqrt{\bar C} \ge 2\mu$ and $2/3\le \alpha \le 1, \forall m\ge 1$ we have
\[
\alpha^{-1}\eta +\tfrac{3-2\alpha}{1-\alpha}L_f+\tfrac{2L_f^2\alpha}{(1-\alpha)\eta}\le (1-\alpha)^{-1}(\mu+2L_f+\tfrac{L_f^2}{\mu}).
\]
Using the above bound in \eqnok{RPDG_s_main1}, we obtain \eqnok{uniform_result1}.
The proofs for the complexity bounds are similar to those in Corollary~\ref{cor_RPDG1} and hence
the details are skipped.
\end{proof}

\vgap

\vgap

Comparing the complexity bounds obtained from Corollaries~\ref{cor_RPDG1} and \ref{cor_RPDG2}
with those of any optimal deterministic first-order method, they differ in a factor of ${\cal O}(\sqrt{mL_f/L})$, 
whenever $\sqrt{mC}\log(1/\epsilon)$ is dominating in \eqnok{non_uniform_com1}.
Clearly, when $L_f$ and $L$ are in the same order of magnitude, RPDG can save up to ${\cal O}(\sqrt{m})$
gradient evaluations for the component function $f_i$ than the deterministic first-order methods. 
However, it should be pointed out that $L_f$ can be much smaller than $L$. In particular, 
when $L_f=L_i, i=1,\ldots,m$, $L_f=L/m$. 
In the next subsection, we will construct examples in 
such extreme cases to obtain the lower complexity bound for general randomized incremental gradient methods.


\subsection{Lower complexity bound for randomized methods}\label{sec_lower_comp}
Our goal in this subsection is to demonstrate that the complexity bounds
obtained in Theorem~\ref{theorem_RPDG_s}, and Corollaries~\ref{cor_RPDG1} and \ref{cor_RPDG2} 
for the RPDG method are essentially not improvable.
Observe that although there exist rich lower complexity bounds in the literature
for deterministic first-order methods (e.g.~\cite{nemyud:83,Nest04}), the study on 
lower complexity bounds for randomized methods are still quite limited.
Recently Agarwal and Bottou~\cite{AgrBott14-1} suggested a lower complexity
bound for minimizing the finite-sum convex optimization problem given in the form of \eqnok{cp}.
However, 
their bounds are developed for deterministic algorithms and hence not applicable to randomized
incremental gradient methods.

To derive the performance limit of the incremental gradient methods, we consider 
a special class of unconstrained and separable strongly convex optimization
problems given in the form of
 \beq \label{worst_problem}
\min_{x_i \in \bbr^{\tilde n}, i = 1, \ldots, m} \left\{\Psi(x) := \tsum_{i=1}^m \left[ f_i(x_i)  + \tfrac{\mu}{2} \|x_i\|_2^2 \right]\right\} .
\eeq
Here $ \tilde n \equiv n / m \in \{1, 2, \ldots \}$ and $\|\cdot\|_2$ denotes standard Euclidean norm.
To fix the notation, we also denote $x = (x_1, \ldots, x_m)$. Moreover, we assume that
$f_i$'s are quadratic functions given by
\beq \label{quadratic_f}
f_i(x_i) = \tfrac{\mu (\cQ-1) }{4} \left[\tfrac{1}{2} \langle A x_i, x_i \rangle - \langle e_1, x_i\rangle \right],
\eeq
where $e_1 := (1, 0, \ldots, 0)$ and $A$ is a symmetric matrix in $\bbr^{\tilde n \times \tilde n}$
given by
\beq \label{defA}
A = \left(
\begin{array}{cccccccc}
2 & -1 & 0 & 0 & \cdots &0 & 0& 0\\
-1& 2  &-1 &0 & \cdots & 0 & 0& 0\\
0& -1 & 2  &-1& \cdots & 0 & 0& 0\\
\cdots&\cdots&\cdots&\cdots&\cdots&\cdots&\cdots\\
0& 0&  0  & 0& \cdots &-1& 2 & -1 \\
0& 0& 0  & 0 & \cdots & 0& -1 & \kappa\\
\end{array}
\right)
\ \ \ \mbox{with} \ \ \
\kappa = \tfrac{\sqrt{\cQ} + 3}{\sqrt{\cQ} + 1}.
\eeq
Compared with the classic worst-case example given in \cite{Nest04},
the tridiagonal matrix $A$ above consists of a different diagonal element $\kappa$ (instead of
$2$). This modification allows us to study problems of finite dimension more conveniently.
It can be easily checked that $A \succeq 0$ and its maximum eigenvalue does not exceeds $4$.
Indeed, for any $s \equiv (s_1, \ldots, s_{\tilde n}) \in \bbr^{\tilde n}$, we have
\begin{align*}
\langle A s, s\rangle 
&= s_1^2 + \tsum_{i=1}^{\tilde n - 1} (s_i - s_{i+1})^2 + (\kappa - 1) s_{\tilde n}^2 \ge 0 \\
\langle A s, s \rangle 
&\le s_1^2 + \tsum_{i=1}^{\tilde n - 1} 2( s_i^2 + s_{i+1}^2) + (\kappa -1) s_{\tilde n}^2  \\
&= 3 s_1^2 + 4 \tsum_{i=2}^{\tilde n - 1} s_i^2 + (\kappa+1) s_{\tilde n}^2 \le 4 \|s\|_2^2, 
\end{align*}
where the last inequality follows from the fact that $\kappa \le 3$. Therefore, for any $\cQ > 1$, the component functions
$f_i$ in \eqnok{quadratic_f} are convex and their gradients are Lipschitz continuous with
constant bounded by $L_i = \mu (\cQ-1)$, $i = 1, \ldots, m$. 

We consider a general class of randomized incremental gradient methods which
sequentially acquire the gradients of a
randomly selected component function $f_{i_t}$ at iteration $t$. More specifically, we assume that
the independent random variables $i_t$, $t = 1, 2, \ldots$, satisfy
\beq \label{prob_rand}
\prob\{i_t = i\} = p_i \ \ \ \mbox{and}  \ \ \ \tsum_{i=1}^m p_i = 1, \ \ p_i \ge 0, i = 1, \ldots, m.
\eeq
Similar to \cite{Nest04}, we assume that these methods generate a sequence of
test points $\{x^k\}$ such that
\beq \label{linear_span}
x^k \in x^0 + { \rm{Lin} }\{\nabla f_{i_1} (x^0), \ldots, \nabla f_{i_k} (x^{k-1}) \},
\eeq
where $\rm {Lin}$ denotes the linear span.

Theorem~\ref{the_lower_bound} below describes the performance limit of
the above randomized incremental gradient methods for solving \eqnok{worst_problem}.

\begin{theorem} \label{the_lower_bound}
Let $x^*$ be the optimal solution of problem \eqnok{worst_problem} and denote
\beq
q := \tfrac{\sqrt{\cQ} - 1}{\sqrt{\cQ} + 1}. \label{def_q1} 
\eeq
Then the iterates $\{x^k\}$ generated by any 
randomized incremental gradient method must satisfy
\beq \label{bnd_rate_below} 
\tfrac{\bbe[\|x^k - x^*\|^2_2]}{\|x^0 - x^*\|^2_2} \ge \tfrac{1}{2} \exp\left(- \tfrac{4 k \sqrt{\cQ} }{m (\sqrt{\cQ} +1)^2 - 4 \sqrt{\cQ}}  \right)
\eeq
for any
\beq  \label{def_N0}
n \ge \underline n(m,k) \equiv  \tfrac{m \log \left[\left(1 - (1 - q^2) / m \right)^k /2 \right]}{2 \log q}.
\eeq
\end{theorem}

\vgap

As an immediate consequence of Theorem~\ref{the_lower_bound}, we
obtain a lower complexity bound for randomized incremental gradient methods.

\begin{corollary} \label{rand_low1}
The number of gradient evaluations performed by any randomized  incremental gradient
methods for finding a solution $\bar x \in X$ of problem \eqnok{cp} such that $\bbe[\|\bar x - x^*\|_2^2 ] \le \epsilon$
cannot be smaller than
\[
\Omega\left\{  \left( \sqrt{ m {\cal C} } + m \right) \log \tfrac{\|x^0 - x^*\|_2^2}{\epsilon} \right\}
\]
if $n$ is sufficiently large, where ${\cal C} = L / \mu$ and $L = \tsum_{i=1}^m L_i$.
\end{corollary}

\begin{proof}
It follows from \eqnok{bnd_rate_below} that the number of iterations $k$ 
required by any randomized incremental gradient methods to find an approximate solution $\bar x$ must satisfy
\beq \label{bnd_k_rnd}
k \ge \left(\tfrac{m (\sqrt{\cQ} +1)^2 } {4 \sqrt{\cQ} } - 1\right) \log \tfrac{\|x^0 - x^*\|_2^2}{2 \epsilon}
\ge \left[\tfrac{m}{2} \left( \tfrac{\sqrt{\cQ}}{2} +1 \right) - 1 \right] \log \tfrac{\|x^0 - x^*\|_2^2}{2 \epsilon}  .
\eeq
Noting that for the worst-case instance in \eqnok{worst_problem}, 
we have $L_i = \mu (\cQ -1)$, $i = 1, \ldots, m$, and hence that
$L = \tsum_{i= 1}^m L_i = m \mu (Q-1)$. Using this relation, we conclude that
\[
k \ge \left[\tfrac{1}{2} \left( \tfrac{\sqrt{ m {\cal C} + m^2}}{2} + m \right) - 1 \right] \log \tfrac{\|x^0 - x^*\|_2^2}{2 \epsilon}
=: \underline k.
\]
The above bound holds when $n \ge  \underline n(m,\underline k)$.
\end{proof}

\vgap

In view of Theorem~\ref{the_lower_bound}, we can also
derive a lower complexity bound for randomized block coordinate descent methods,
which update one randomly selected block of variables at each iteration for $\min_{x \in X} \Psi(x)$. Here
$\Psi$ is smooth and strongly convex such that
\[
\tfrac{\mu_\Psi}{2} \|x - y\|_2^2 \le \Psi(x) - \Psi(y) - \langle \nabla \Psi(y), x- y \rangle \le \tfrac{L_\Psi}{2} \|x - y\|_2^2, \forall x, y \in X.
\]

\begin{corollary} \label{rand_low2}
The number of iterations performed by any randomized  block coordinate descent methods
for finding a solution $\bar x \in X$ of $\min_{x \in X} \Psi(x)$ such that $\bbe[\|\bar x - x^*\|_2^2 ] \le \epsilon$
cannot be smaller than
\[
\Omega\left\{  \left( m \sqrt{ \cQ_\Psi} \right) \log \tfrac{\|x^0 - x^*\|_2^2}{\epsilon} \right\}
\]
if $n$ is sufficiently large, where $\cQ_\Psi = L_\Psi /\mu_\Psi$ denotes the condition number of $\Psi$.
\end{corollary}

\begin{proof}
The worst-case instances in \eqnok{worst_problem} have a block separable structure. Therefore,
any randomized incremental gradient methods are equivalent to randomized block coordinate descent methods.
The result then immediately follows from \eqnok{bnd_k_rnd}.
\end{proof}


\setcounter{equation}{0}
\section{Generalization of randomized primal-dual gradient methods} \label{sec_general}
In this section, we generalize the RPDG method for solving a few different types of 
convex optimization problems which are not necessarily smooth 
and strongly convex.

\subsection{Smooth problems with bounded feasible sets} \label{sec_smooth_bnd}
Our goal in this subsection is to generalize RPDG for solving smooth problems without strong convexity (i.e., $\mu = 0$).
Different from the deterministic PDG method, it is difficult to develop a simple stepsize
policy for $\{\tau_t\}$, $\{\eta_t\}$, and $\{\alpha_t\}$ which can guarantee 
the convergence of this method unless a weaker termination criterion is used (see \cite{DangLan14-1}).
In order to obtain stronger convergence results, we will discuss a different approach obtained by applying the RPDG 
method to a slightly perturbed problem 
of \eqnok{cp}. 

In order to apply this perturbation approach, we will assume that $X$ is bounded (see Subsection~\ref{sec_unconstrained_nonsmooth}
for possible extensions), i.e., given $x_0\in X$, $\exists \Omega_X \ge 0$ s.t.
\beq \label{def_omega}
\max_{x \in X} P_\w(x_0, x) \le \Omega_X^2.
\eeq
Now we define the perturbation problem as
\beq \label{def_perturb}
\Psi_\delta^* := \min_{x \in X} \left\{ \Psi_\delta(x) := f(x) + h(x) + \delta P_\w(x_0, x)\right\},
\eeq
for some fixed $\delta > 0$. It is well-known that an approximate solution of \eqnok{def_perturb} will also be
an approximate solution of \eqnok{cp} if $\delta$ is sufficiently small. More specifically, it is easy to verify that
\begin{align} 
\Psi^* \le &\Psi_\delta^* \le \Psi^* + \delta \Omega_X^2, \label{relate_perturb1}\\
\Psi(x) \le &\Psi_\delta(x) \le \Psi(x) + \delta \Omega_X^2, \ \ \forall x \in X. \label{relate_perturb2}
\end{align}
 
The following result describes the complexity associated with this perturbation approach
for solving smooth problems without strong convexity (i.e., $\mu = 0$).

\begin{proposition} \label{prop_smooth}
Let us apply the RPDG method with the parameter settings in Corollary~\ref{cor_RPDG1}
to the perturbation problem \eqnok{def_perturb} with
\beq \label{def_delta}
\delta = \tfrac{\epsilon}{2 \Omega_X^2}, 
\eeq
for some $\epsilon > 0$.  Then we can find a solution $\bar x \in X$ s.t. 
$\bbe[\Psi(\bar x) - \Psi^*] \le \epsilon$ in at most 
\beq
{\cal O} \left\{\left(m + \sqrt{\tfrac{m L \Omega_X^2}{\epsilon}} \right) \log \tfrac{mL_f \Omega_X}{\epsilon} \right\}\label{non_strong_com1}
\eeq
iterations. Moreover, we can find a solution $\bar x \in X$ s.t.
$\Prob\{\Psi(\bar x) - \Psi^* > \epsilon \} \le \lambda$ for any $\lambda \in (0,1)$ in at most
\beq
{\cal O} \left\{\left(m + \sqrt{\tfrac{m L \Omega_X^2}{\epsilon}} \right) \log \tfrac{mL_f \Omega_X}{\lambda \epsilon}\right\} \label{non_strong_com2}
\eeq
iterations.
\end{proposition}

\begin{proof}
Let $x^*_\delta$ be the optimal solution of \eqnok{def_perturb}.
Denote $C := 16 L \Omega_X^2 / \epsilon$ and
\[
K := 2\left[(m+1) + \sqrt{(m-1)^2 + 4 m C} \right] \log \left[(m+\sqrt{mC})(\delta+2L_f+\tfrac{L_f^2}{\delta})\tfrac{4 \Omega_X^2}{\epsilon}\right].
\]
It can be easily seen that 
\begin{align*}
\Psi(\bar x^K) - \Psi^* 
&\le \Psi_{\delta}(\bar x^K) - \Psi_\delta^* + \delta \Omega_X^2 
 =\Psi_{\delta}(\bar x^K) - \Psi_\delta^*  + \tfrac{\epsilon}{2}.
\end{align*}
Note that problem \eqnok{def_perturb} is given in the form of \eqnok{cp} with the strongly convex modulus
$\mu = \delta$, and $h(x) = h(x) - \delta \langle \w'(x_0), x \rangle$.
Hence by applying Corollary~\ref{cor_RPDG1}, we have
\[
\bbe[\Psi_{\delta}(\bar x^{K}) - \Psi_\delta^*] 
\le \tfrac{\epsilon}{2}.
\]
Combining these two inequalities, we have $\bbe[\Psi(\bar x^K) - \Psi^*] \le \epsilon$, which implies the bound in \eqnok{non_strong_com1}.
The bound in \eqnok{non_strong_com2} can be shown similarly and hence the details are skipped.
\end{proof}

\vgap

Observe that if we apply a deterministic optimal first-order method (e.g., Nesterov's method or the PDG method), the total number of gradient evaluations for 
$\nabla f_i$, $i = 1, \ldots, m$, would be given by
\[
m \sqrt{\tfrac{L_f \Omega_X^2}{\epsilon}}.
\]
Comparing this bound with \eqnok{non_strong_com1}, we can see that the number of gradient evaluations performed 
by the RPDG method can be ${\cal O} \left(\sqrt {m}\log^{-1} (mL_f \Omega_X/\epsilon)\right)$ times
smaller than these deterministic methods when $L$ and $L_f$ are in the same order of magnitude.

\subsection{Structured nonsmooth problems}
In this subsection, we assume that the smooth components $f_i$ are nonsmooth but can be approximated closely by
smooth ones. More specifically, we assume that
\beq \label{nonsmooth_f}
f_i(x) := \max_{y_i \in Y_i} \langle A_i x, y_i\rangle - q_i(y_i).  
\eeq
Nesterov in an important work~\cite{Nest05-1} shows that we can approximate $f_i(x)$ and $f$, respectively, by
\beq \label{approximat_f}
\tilde f_i(x, \delta) := \max_{y_i \in Y_i} \langle A_i x, y_i\rangle - q_i(y_i) - \delta v_i(y_i) \ \ \mbox{and} \ \ \tilde f(x, \delta) = \tsum_{i=1}^m \tilde f_i(x, \delta),
\eeq
where $v_i(y_i)$ is a strongly convex function with modulus $1$ such that
\beq
0 \le v_i(y_i) \le \Omega_{Y_i}^2, \ \ \ \forall y_i \in Y_i.
\eeq
In particular, we can easily show that
\beq \label{closeness}
\tilde f_i(x, \delta) \le f_i(x) \le \tilde f_i(x, \delta) + \delta \Omega_{Y_i}^2 \ \ \mbox{and} \ \ \tilde f(x, \delta) \le f(x) \le \tilde f(x, \delta) + \delta \Omega_Y^2,
\eeq
for any $x \in X$, where $\Omega_Y^2 = \tsum_{i=1}^m \Omega_{Y_i}^2$.
Moreover, $f_i(\cdot, \delta)$ and $f(\cdot, \delta)$ are continuously differentiable and their gradients are Lipschitz continuous with
constants given by
\beq \label{Lipshitz_tf}
\tilde L_i = \frac{\|A_i\|^2}{\delta} \ \ \ \mbox{and} \ \ \ \tilde L = \tfrac{\tsum_{i=1}^m \|A_i\|^2}{\delta} = \tfrac{\|A\|^2}{\delta},
\eeq
respectively. 
As a consequence, we can apply the RPDG method to solve the approximation problem
\beq \label{app_cp}
\tilde \Psi^*_\delta := \min_{x \in X} \left\{ \tilde \Psi_\delta(x) := \tilde f(x, \delta) + h(x) + \mu \w(x)\right\}.
\eeq

The following result provides complexity bounds of the RPDG method 
for solving the above structured nonsmooth problems for the case when $\mu > 0$.

\begin{proposition} \label{prop_nonsmooth1}
Let us apply the RPDG method with the parameter settings in Corollary~\ref{cor_RPDG1}
to the approximation problem \eqnok{app_cp} with 
\beq \label{def_delta_smoothing}
\delta = \tfrac{\epsilon}{2 \Omega_Y^2},
\eeq
for some $\epsilon > 0$. 
Then we can find a solution $\bar x \in X$ s.t. $\bbe[\Psi(\bar x) - \Psi^*] \le \epsilon$ in
at most
\beq \label{smoothing_bnd1}
{\cal O} \left\{ \|A\| \Omega_Y \sqrt{\tfrac{m }{\mu \epsilon}} \log \tfrac{m\|A\| \Omega_X \Omega_Y}{\mu \epsilon}\right\}
\eeq 
iterations.
Moreover, we can find a solution $\bar x \in X$
s.t. $\Prob\{ \Psi(\bar x) - \Psi^*> \epsilon\} \le \lambda$ for any $\lambda \in (0,1)$ in at most
\beq \label{smoothing_bnd2}
{\cal O} \left\{
 \|A\| \Omega_Y \sqrt{\tfrac{m }{\mu \epsilon}} \log \tfrac{m\|A\| \Omega_X \Omega_Y}{\lambda \mu \epsilon}\right\}
\eeq
iterations.
\end{proposition}

\begin{proof}
It follows from \eqnok{closeness} and \eqnok{app_cp} that
\beq \label{app_nontemp}
\Psi(\bar x^k) - \Psi^* 
\le \tilde \Psi_\delta(\bar x^k) - \tilde \Psi_\delta^* + \delta \Omega_Y^2
 = \tilde \Psi_\delta(\bar x^k) - \tilde \Psi_\delta^* + \tfrac{\epsilon}{2}.
\eeq
Using relation \eqnok{Lipshitz_tf} and Corollaries~\ref{cor_RPDG1}, we conclude that a solution 
$\bar x^k \in X$ satisfying $\bbe[ \tilde \Psi_\delta(\bar x^k) - \tilde \Psi_\delta^* ] \le \epsilon/2$
can be found in 
\[
{\cal O} \left\{ \|A\| \Omega_Y \sqrt{\tfrac{m }{\mu \epsilon}} \log \left[(m+\sqrt{\tfrac{m\tilde L}{\mu}})\left(\mu+ 2\tilde L + \tfrac{\tilde L^2}{\mu}\right) \tfrac{\Omega_X^2 }{\epsilon}\right] \right\}
\]
iterations. This observation together with \eqnok{app_nontemp} and the definition of $\tilde L$ in \eqnok{Lipshitz_tf}
then imply the bound in \eqnok{smoothing_bnd1}.
The bound in \eqnok{smoothing_bnd2} follows similarly from \eqnok{app_nontemp} and
Corollaries~\ref{cor_RPDG1}, and hence the details are skipped.
\end{proof}

\vgap

The following result holds for  the RPDG method 
applied to the above structured nonsmooth problems when $\mu = 0$.
\begin{proposition} \label{prop_nonsmooth2}
Let us apply the RPDG method with the parameter settings in Corollary~\ref{cor_RPDG1}
to the approximation problem \eqnok{app_cp} with $\delta$ in \eqnok{def_delta_smoothing}
for some $\epsilon > 0$.
Then we can find a solution $\bar x \in X$ s.t. $\bbe[\Psi(\bar x) - \Psi^*] \le \epsilon$ in at most
\[
{\cal O}\left\{\tfrac{\sqrt{m} \|A\| \Omega_X \Omega_Y}{\epsilon} \log \tfrac{m\|A\| \Omega_X \Omega_Y}{\epsilon} \right\}
\] 
iterations.
Moreover, we can find a solution $\bar x \in X$ 
s.t. $\Prob\{ \Psi(\bar x) - \Psi^*> \epsilon\} \le \lambda$ for any $\lambda \in (0,1)$ in at most
\[
{\cal O}\left\{\tfrac{\sqrt{m} \|A\| \Omega_X \Omega_Y}{\epsilon} \log \tfrac{m\|A\| \Omega_X \Omega_Y}{\lambda \epsilon} \right\}
\] 
iterations.
\end{proposition}

\begin{proof}
Similarly to the arguments used in the proof of Proposition~\ref{prop_nonsmooth1},
our results follow from \eqnok{app_nontemp}, and an application of 
Proposition~\ref{prop_smooth} to problem~\eqnok{app_cp}.
\end{proof}

\vgap

By Propositions~\ref{prop_nonsmooth1} and \ref{prop_nonsmooth2}, the total number of 
gradient computations for $\tilde f(\cdot, \delta)$ performed by the RPDG method,
after disregarding the logarithmic factors, can be ${\cal O}(\sqrt{m})$ times smaller than
those required by deterministic first-order methods, such as Nesterov's smoothing technique~\cite{Nest05-1}.

\subsection{Unconstrained smooth problems} \label{sec_unconstrained_nonsmooth}

In this subsection, we set $X = \bbr^n$, $h(x) = 0$, and $\mu = 0$ in \eqnok{cp} and
consider the basic convex programming problem of 
\beq \label{basic_cp}
f^* := \min_{x \in \bbr^n} \left\{ f(x) := \tsum_{i=1}^m f_i(x)\right\}.
\eeq
We assume that the set of optimal solutions $X^*$ of this problem is nonempty.

We will still use the perturbation-based approach as described in Subsection~\ref{sec_smooth_bnd} by
solving the perturbation problem given by
\beq \label{basic_cp_pert}
f^*_\delta := \min_{x \in \bbr^n} \left\{f_\delta(x) := f(x) + \tfrac{\delta}{2} \|x -x^0\|_2^2, \right\}
\eeq
for some $x^0\in X, \delta > 0$, where $\|\cdot\|_2$ denotes the Euclidean norm.
Also let $L_\delta$ denote the Lipschitz constant for $f_\delta(x)$. Clearly, $L_\delta=L+\delta$.
Since the problem is unconstrained and the information on the size of the optimal solution is unavailable,
it is hard to estimate the total number of iterations by using the absolute accuracy in terms of
$\bbe[f(\bar x) - f^*]$. Instead, we define the relative accuracy associated with a given 
$\bar x \in X$ by
\beq \label{def_rel_acc}
\rec(\bar x, x^0, f^*) := \tfrac{2 [f(\bar x) - f^*]}{L (1+\min_{u \in X^*} \|x^0 - u\|_2^2)}.
\eeq

We are now ready to establish the complexity of the RPDG method applied to \eqnok{basic_cp}
in terms of $\rec(\bar x, x^0, f^*)$. 
\begin{proposition} \label{prop_basic_cp}
Let us apply the RPDG method with the parameter settings in Corollary~\ref{cor_RPDG1}
to the perturbation problem \eqnok{basic_cp_pert} with 
\beq \label{delta_basic_cp}
\delta = \tfrac{L \epsilon}{2},
\eeq
for some $\epsilon > 0$. Then we can find a solution $\bar x \in X$ s.t.
$\bbe[\rec(\bar x, x^0, f^*)] \le \epsilon$ in at most
\beq \label{basic_cp_bnd1}
{\cal O}\left\{ \sqrt{\tfrac{m}{\epsilon}} \log \tfrac{m}{\epsilon}\right\}
\eeq
iterations. Moreover, we can find a solution $\bar x \in X$ 
s.t. $\Prob\{\rec(\bar x, x^0, f^*)> \epsilon\} \le \lambda$ for any $\lambda \in (0,1)$ in at most
\beq \label{basic_cp_bnd2}
{\cal O} \left\{ \sqrt{\tfrac{m}{\epsilon}} \log \tfrac{m}{\lambda \epsilon}\right\}
\eeq
iterations.
\end{proposition}

\begin{proof}
Let $x^*_\delta$ be the optimal solution of \eqnok{basic_cp_pert}.
Also let $x^*$ be the optimal solution of \eqnok{basic_cp} that is closest to $x^0$, i.e.,
$x^* = \argmin_{u \in X^*}\|x^0 - u\|_2$.
It then follows from the strong convexity of $f_\delta$ that
\begin{align}
\tfrac{\delta}{2} \|x_\delta^* - x^*\|_2^2 
&\le f_\delta(x^*) - f_\delta(x_\delta^*) \nn\\
&= f(x^*) + \tfrac{\delta}{2} \|x^* - x^0\|_2^2 - f_\delta(x^*_\delta) \nn\\
&\le \tfrac{\delta}{2} \|x^* - x^0\|_2^2,\nn
\end{align}
which implies that
\beq \label{closeness_basic_cp}
\|x_\delta^* - x^*\|_2 \le \|x^* - x^0\|_2.
\eeq
Moreover, using the definition of $f_\delta$ and the fact that
$x^*$ is feasible to \eqnok{basic_cp_pert}, we have
\[
f^* \le f^*_\delta \le f^* +\tfrac{\delta}{2} \|x^* - x^0\|_2^2,
\]
which implies that
\begin{align*}
f(\bar x^K) - f^* 
&\le f_\delta(\bar x^K) - f_\delta^* + f_\delta^* - f^* \nn\\
&\le f_\delta(\bar x^K) - f_\delta^* + \tfrac{\delta}{2} \|x^* - x^0\|_2^2.
\end{align*}
Now suppose that we run the RPDG method applied to \eqnok{basic_cp_pert} for $K$ iterations. Then
by Corollary~\ref{cor_RPDG1}, we have
\begin{align*}
\bbe[f_\delta(\bar x^K)-f_\delta^*] 
&\le \alpha^{K/2} (1-\alpha)^{-1}\left(\delta + 2L_\delta + \tfrac{L_\delta^2}{\delta} \right) \|x^0- x^*_\delta\|_2^2 \nn\\
&\le \alpha^{K/2} (1-\alpha)^{-1}\left(\delta + 2L_\delta + \tfrac{L_\delta^2}{\delta} \right) [\|x^0 - x^*\|_2^2 + \|x^* - x^*_\delta\|_2^2] \nn\\
&= 2 \alpha^{K/2} (1-\alpha)^{-1}\left( 3\delta + 2L + \tfrac{(L+\delta)^2}{\delta} \right) \|x^0 - x^*\|_2^2,
\end{align*}
where the last inequality follows from \eqnok{closeness_basic_cp} and $\alpha$ is defined in \eqnok{nonuniform_stepsize}
with $C= 8 L_\delta / \delta = \tfrac{8(L+\delta)}{\delta}=8(2 /\epsilon + 1)$.
Combining the above two relations, we have
\[
\bbe[f(\bar x^K) - f^*] \le \left[ 2 \alpha^{K/2} (1-\alpha)^{-1} \left( 3\delta+ 2L + \tfrac{(L+\delta)^2}{\delta} \right) + \tfrac{\delta}{2}\right] [\|x^0 - x^*\|_2^2.
\]
Dividing both sides of the above inequality by $L (1+ \|x^0 - x^*\|_2^2) /2$, we obtain
\begin{align*}
\bbe[\rec(\bar x^K, x^0, f^*)] 
&\le \tfrac{2}{L} \left[ 2 \alpha^{K/2} (1-\alpha)^{-1} \left( 3\delta+ 2L + \tfrac{(L+\delta)^2}{\delta} \right) + \tfrac{\delta}{2}\right] \nn\\
& \le 4\left(m + 2\sqrt{2m(\tfrac{2}{\epsilon}+1)}\right)\left( 3\epsilon+ 4+ (2+\epsilon)(\tfrac{2}{\epsilon}+1)\right) \alpha^{K/2} +\tfrac{\epsilon}{2},
\end{align*}
which clearly implies the bound in \eqnok{basic_cp_bnd1}. The bound in \eqnok{basic_cp_bnd2} also
follows from the above inequality and the Markov's inequality.
\end{proof}

\vgap

By Proposition~\ref{prop_basic_cp}, the total number of gradient evaluations for the component functions
$f_i$ required by the RPDG method can be ${\cal O}(\sqrt{m} \log^{-1}(m/\epsilon))$ times smaller than
those performed by deterministic optimal first-order methods.

\setcounter{equation}{0}
\section{Complexity analysis} \label{sec_analysis}
Our main goal in this section is to prove the main theorems in Sections~\ref{sec_det} and \ref{sec_ORIG}.
After introducing some basic tools and general results about PDG and RPDG methods 
in Subsection~\ref{sec_tools} and \ref{sec_gresults}, respectively, we provide the proofs for 
Theorem~\ref{cor_det1} and Theorem~\ref{theorem_RPDG_s}, which describe the main convergence properties 
for the PDG and RPDG methods, in Subsection~\ref{sec_conver}. 
Moreover, in Subsection~\ref{sec_lower_bnd}, we provide the proof for the lower complexity bound in Theorem~\ref{the_lower_bound}.

\subsection{Some basic tools}\label{sec_tools}
The following result provides a few different bounds on the diameter of
the dual feasible sets ${\cal G}$ and ${\cal Y}$ in \eqnok{cpsaddle_f} and \eqnok{cpsaddle}.

\begin{lemma}
Let $x^0 \in X$ be given, $y^0_i = \nabla f_i(x^0)$, $i = 1, \ldots, m$, and $g^0 = \nabla f(x^0)$. 
Assume that $J_i'(y^0) = x^0$ and $J_f'(g^0) = x^0$ in the definition of $D(y^0, y)$
and $D_f(g^0, g)$ in \eqnok{def_dual_prox} and \eqnok{def_Df}, respectively.
\begin{itemize}
\item [a)] For any  $x \in X$ and $y_i = \nabla f_i(x)$, $i = 1, \ldots, m$, we have
\beq \label{bound_Dy1}
D(y^0, y) \le \frac{L_f}{2} \|x^0 - x\|^2 \le L_f P(x^0, x).
\eeq
\item [b)] If $x^* \in X$ is an optimal solution of \eqnok{cp} and $y^*_i = \nabla f_i(x^*)$, $i = 1, \ldots, m$,
then 
\beq \label{bound_Dy}
D(y^0, y^*) \le \Psi(x^0) - \Psi(x^*).
\eeq
\item [c)] For any $x \in X$ and $g = \nabla f(x)$, we have
\beq \label{bound_Dy0}
D_f(g^0,g) \le \frac{L_f}{2} \|x^0- x\|^2.
\eeq
\end{itemize}
\end{lemma}

\begin{proof}
We first show part a). It follows from the definition of $J_i$, \eqnok{def_dual_prox}, and \eqnok{def_D} that
\begin{align}
D(y^0, y) 
&= J(y) - J(y^0) - \tsum_{i=1}^m\langle J_i'(y^0), y_i-y_i^0 \rangle \nn \\
&= \langle x, U y \rangle - f(x) + f(x^0) - \langle x^0, U y^0 \rangle
- \langle x^0, U(y - y^0) \rangle \nn\\
&= f(x^0) - f(x) - \langle U y, x^0 - x\rangle \nn\\
&\le \frac{L_f}{2} \|x^0 - x\|^2 \le L_f P(x^0, x), \nn
\end{align}
where the last inequality follows from \eqnok{strong_h}.
We now show part b).  
By the above relation, the convexity of $h$ and $\w$, and the optimality of $(x^*,y^*)$, we
have
\begin{align}
D(y^0, y^*) &= f(x^0) - f(x^*) - \langle U y^*, x^0 - x^*\rangle \nn\\
&=f(x^0) - f(x^*) + \langle h'(x^*)+\mu \w'(x^*), x^0-x^*\rangle - \langle U y^* + h'(x^*) + \mu \w'(x^*), x^0 - x^* \rangle \nn \\ 
&\le f(x^0) - f(x^*) + \langle h'(x^*) + \mu \w'(x^*), x^0-x^*\rangle \le \Psi(x^0) - \Psi(x^*).\nn
\end{align}
The proof of part c) is similar to part a) and hence the details are skipped.
\end{proof}

\vgap

The following lemma gives an important bound for the primal optimality gap $\Psi(\bar{x})-\Psi(x^*)$
for some $\bar x \in X$.
\begin{lemma}
Let $(\bar{x},\bar{y})\in Z$ be a given pair of feasible solutions of \eqnok{cpsaddle}, 
and $z^*=(x^*,y^*)$ be a pair of optimal solutions of \eqnok{cpsaddle}. Then, we have
\beq\label{psi_Q}
\Psi(\bar{x})-\Psi(x^*) \le Q((\bar{x},\bar{y}),z^*)+ \frac{L_f}{2}\|\bar{x}-x^*\|^2.
\eeq
\end{lemma}

\begin{proof}
Let $\bar{y}_*=(\nabla f_1(\bar{x});\nabla f_2(\bar{x});\dots;\nabla f_m(\bar{x}))$, 
and by the definition of $Q(\cdot,\cdot)$ in \eqnok{def_gap}, we have
\begin{align*}
Q((\bar{x},\bar{y}),z^*)
&=\left[h(\bar x) + \mu \w(\bar x) + \langle \bar x, U y^* \rangle -  J(y^*)\right]
- \left[\cX(x^*) + \mu \w(x^*) + \langle x^*, U \bar y \rangle - J(\bar y) \right]\\
&\ge \left[h(\bar{x})+\mu\w(\bar{x})+\langle \bar{x}, U\bar{y}_*\rangle - J(\bar{y}_*)\right]
+\langle \bar{x}, U(y^*-\bar{y}_*)\rangle-J(y^*)+J(\bar{y}_*)\\
&\quad -\left[\cX(x^*)+\mu \w(x^*)+\max_{y\in \mathcal{Y}}\left\{\langle x^*, Uy\rangle-J(y)\right\}\right]\\
&=\Psi(\bar x)-\Psi(x^*)+\langle \bar{x}, U(y^*-\bar{y}_*)\rangle -\langle x^*,Uy^*\rangle + f(x^*)+\langle \bar{x}, U \bar{y}_*\rangle -f(\bar{x})\\
&=\Psi(\bar x)-\Psi(x^*)+f(x^*)-f(\bar{x})+\langle \bar{x}-x^*,\nabla f(x^*)\rangle
\ge \Psi(\bar x)-\Psi(x^*)-\frac{L_f}{2}\|\bar x-x^*\|^2,
\end{align*}
where the second equality follows from the fact that $J_i,i=1,\dots,m$, are the conjugate functions of $f_i$.  
\end{proof}

\subsection{General results for both PDG and RPDG}\label{sec_gresults}
We will establish some general convergence results in Proposition~\ref{main_theorm}
which holds 
for both deterministic and randomized PDG methods by viewing
PDG as a special case of RPDG with $m=1$. Then both Theorems~\ref{cor_det1}
and \ref{theorem_RPDG_s} follow as some immediate consequences of Proposition~\ref{main_theorm}.

Before showing Proposition~\ref{main_theorm}
we will develop a few technical results. Lemma~\ref{tech1_prox} below characterizes
the solutions of the prox-mapping in \eqnok{prox_mapping} and \eqnok{d_prox_mapping}.
This result generalizes some previous results (e.g., Lemma 6 of \cite{LaLuMo11-1} and Lemma 2 of \cite{GhaLan12-2a}). 

\begin{lemma} \label{tech1_prox}
Let $U$ be a closed convex set and a point $\tilde u \in U$ be given. 
Also let $w:U \to \bbr$ be a convex function and
\beq \label{def_bregman_w}
W(\tilde u, u) = w(u) - w(\tilde u) - \langle w'(\tilde u), u - \tilde u \rangle,
\eeq
for some $w'(\tilde u) \in \partial w(\tilde u)$.
Assume that the function $q: U \to \bbr$ satisfies
\beq \label{strong_q}
q(u_1) - q(u_2) - \langle q'(u_2), u_1 - u_2 \rangle \ge \mu_0 W(u_2, u_1), \ \ \forall u_1, u_2 \in U
\eeq
for some $\mu_0 \ge 0$.
Also assume that the scalars $\mu_1$ and $\mu_2$ 
are chosen such that $\mu_0 + \mu_1 + \mu_2 \ge 0$.
If
\beq \label{tech_lemma_prob}
u^* \in \Argmin \{ q(u) + \mu_1 w(u) + \mu_2 W(\tilde u, u) : u \in U\},
\eeq
then for any $ u \in U$, we have
\[
q(u^*) + \mu_1 w(u^*) + \mu_2 W(\tilde u, u^*) + (\mu_0 + \mu_1 + \mu_2) W(u^*,u) 
 \le q(u) + \mu_1 w(u) + \mu_2 W(\tilde u, u).
\]
\end{lemma}

\begin{proof}
Let $\phi(u) := q(u) + \mu_1 w(u) + \mu_2 W(\tilde u, u)$.
It can be easily checked that for any $u_1, u_2 \in U$, 
\begin{align*}
W(\tilde u, u_1) &= W(\tilde u, u_2) + \langle W'(\tilde u, u_2), u_1-u_2 \rangle + W(u_2, u_1), \\
w(u_1) &= w(u_2) + \langle w'(u_2), u_1 - u_2 \rangle + W(u_2,u_1).
\end{align*}
Using these relations and \eqnok{strong_q}, we conclude that
\beq \label{tech_tt}
\phi(u_1) - \phi(u_2) - \langle \phi'(u_2), u_1 - u_2 \rangle \ge (\mu_0 + \mu_1 + \mu_2) W(u_2, u_1)
\eeq
for any $u_1, u_2 \in Y$, which together with the fact that $\mu_0+\mu_1 + \mu_2 \ge 0$ then imply
that $\phi$ is convex. Since $u^*$ is an optimal solution of \eqnok{tech_lemma_prob},
we have $\langle \phi'(u^*), u - u^* \rangle \ge 0$. Combining this inequality with \eqnok{tech_tt},
we conclude that 
\[
\phi(u) - \phi(u^*) \ge (\mu_0 + \mu_1 + \mu_2) W(u^*, u), 
\]
from which the result immediately follows.
\end{proof}

\vgap

The following simple result provides a few identities related to $y^t$ and $\tilde y^t$ that 
will be useful for the analysis of the PDG algorithm.
\begin{lemma}
Let $y^t$, $\tilde y^t$, and $\hat y^t$ be defined in \eqnok{def_yt_r}, \eqnok{def_tyt_r}, 
and \eqnok{def_hyt_r}, respectively. Then we have, for any $i = 1, \ldots, m$ and $t = 1, \ldots, k$,
\begin{align}
\bbe_t[D_i(y^{t-1}_i, y_i^t)] &= p_i D_i(y^{t-1}_i, \hat y_i^t), \label{exp_Dyt}\\
\bbe_t[D_i(y_i^t, y_i)] &= p_i D_i(\hat y^t_i, y_i) + (1-p_i) D_i(y^{t-1}_i, y_i), \label{exp_Dyy}
\end{align}
for any $y \in {\cal Y}$,
where $\bbe_t$ denotes the conditional expectation w.r.t. $i_t$ given $i_1, \ldots, i_{t-1}$.
\end{lemma}

\begin{proof}
\eqnok{exp_Dyt} follows immediately from
the facts that $\Prob_t \{y_i^t = \hat y_i^t\} = \Prob_t\{ i_t = i\}
= p_i$ and $\Prob_t \{y_i^t = y_i^{t-1}\} = 1 - p_i$. Here $\Prob_t$ denotes the
conditional probability w.r.t. $i_t$ given $i_1, \ldots, i_{t-1}$.
Similarly, we can show \eqnok{exp_Dyy}.
\end{proof}

\vgap

We now prove an important recursion about the RPDG method.

\begin{lemma}
Let the gap function $Q$ be defined in \eqnok{def_gap}. Also
let $x^t$ and $\hat y^t$ be defined in \eqnok{def_xt_r} and \eqnok{def_hyt_r}, respectively.
Then for any $t \ge 1$, we have
\begin{align}
\bbe[Q((x^t, \hat y^t), z)] & \le \bbe\left[ \eta_t P(x^{t-1}, x) - (\mu + \eta_t) P(x^t, x)
- \eta_t P(x^{t-1}, x^t) \right] \nn\\
& \quad + \tsum_{i=1}^m\bbe \left[ \left(p_i^{-1} (1+\tau_t)-1\right) D_i(y^{t-1}_i, y_i)- p_i^{-1} (1+ \tau_t) D_i(y^{t}_i, y_i)\right] \nn\\
& \quad + \bbe\left[ \langle \tilde x^t - x^t, U(\tilde y^t - y) \rangle - \tau_t p_{i_t}^{-1} D_{i_t}(y_{i_t}^{t-1}, y_{i_t}^t)\right], \ \ \ \forall z \in Z.\label{ORIG_rec}
\end{align}
\end{lemma}

\begin{proof}
It follows from Lemma~\ref{tech1_prox} applied to \eqnok{def_xt_r} that $\forall x \in X$,
\beq \label{primal_opt_r}
\langle x^t - x, U \tilde y^t \rangle + h(x^t) + \mu \w(x^t)  - h(x) - \mu \w(x) \le
\eta_t P(x^{t-1}, x) - (\mu + \eta_t) P(x^t, x) - \eta_t P(x^{t-1}, x^t).
\eeq
Moreover, by Lemma~\ref{tech1_prox} applied to \eqnok{def_hyt_r}, we have, for any 
$i=1, \ldots, m$ and $t =1, \ldots, k$,
\[
\langle - \tilde x^t, \hat y^t_i - y_i \rangle + J_i(\hat y_i^t) - J_i(y_i) \le \tau_t D_i(y^{t-1}_i, y_i)
- (1+ \tau_t) D_i(\hat y^{t}_i, y_i) - \tau_t D_i(y^{t-1}_i, \hat y^t_i).
\]
Summing up these inequalities over $i = 1, \ldots, m$, we have, $\forall y \in {\cal Y}$,
\beq \label{dual_opt_r}
\langle - \tilde x^t, U (\hat y^t - y)\rangle + J(\hat y^t) - J(y) \le 
 \tsum_{i=1}^m \left[ \tau_t D_i(y_i^{t-1}, y_i)
- (1+ \tau_t) D_i(\hat y^{t}_i, y_i) - \tau_t D_i(y^{t-1}_i, \hat y^t_i)\right]. 
\eeq
Using the definition of $Q$ in \eqnok{def_gap}, \eqnok{primal_opt_r}, and \eqnok{dual_opt_r},
we have
\begin{align}
Q((x^t, \hat y^t), z) & \le \eta_t P(x^{t-1}, x) - (\mu + \eta_t) P(x^t, x) - \eta_t P(x^{t-1}, x^t) \nn\\
& \quad + \tsum_{i=1}^m \left[ \tau_t D_i(y^{t-1}_i, y_i)
- (1+ \tau_t) D_i(\hat y^{t}_i, y_i) - \tau_t D_i(y^{t-1}_i, \hat y^t_i) \right] \nn\\
& \quad + \langle \tilde x^t, U(\hat y^t - y) \rangle - \langle x^t, U (\tilde y^t - y) \rangle + \langle x, U(\tilde y^t - \hat y^t)\rangle. \label{ORIG_rec1} 
\end{align}
Also observe that by \eqnok{def_yt_r}, \eqnok{exp_tyt}, \eqnok{exp_Dyt},
and \eqnok{exp_Dyy},
\begin{align*}
D_i(y^{t-1}_i, \hat y^t_i) &= 0, \ \forall i \neq i_t,\\
\bbe[ \langle x, U(\tilde y^t - \hat y^t)\rangle] &= 0, \\
\bbe[\langle \tilde x^t, U \hat y^t \rangle] &= \bbe[\langle \tilde x^t, U \tilde y^t \rangle],\\
\bbe[D_i(y^{t-1}_i, \hat y^t_i)] &= \bbe[p_i^{-1} D_i(y^{t-1}_i, y^t_i)]  \\
\bbe[D_i(\hat y^{t}_i, y_i)] &= p_i^{-1} \bbe[D_i(y^{t}_i, y_i)] - (p_i^{-1}-1) \bbe[D_i(y^{t-1}_i, y_i)],
\end{align*}
Taking expectation on both sides of \eqnok{ORIG_rec1} and using the above observations, 
we obtain \eqnok{ORIG_rec}.
\end{proof}

\vgap

We are now ready to establish a general convergence result which holds
for both PDG and RPDG.
\begin{proposition} \label{main_theorm}
Suppose
that $\{\tau_t\}$, $\{\eta_t\}$, and $\{\alpha_t\}$ in the RPDG method satisfy
\begin{align}
\theta_t \left(p_i^{-1} (1+\tau_t)-1\right)  &\le p_i^{-1}  \theta_{t-1} (1+ \tau_{t-1}) , i = 1, \ldots,m; t = 2, \ldots, k, \label{cond_tau}\\
\theta_t \eta_t  &\le  \theta_{t-1} (\mu + \eta_{t-1}) , t = 2, \ldots, k, \label{cond_eta}\\
 \tfrac{\eta_k}{4} &\ge \tfrac{L_i (1- p_i)^2}{\tau_k p_i}, i = 1, \ldots, m,\label{eta_tau_cond1} \\
\tfrac{\eta_{t-1}}{2} &\ge \tfrac{L_i \alpha_t }{\tau_t p_i} 
 + \tfrac{ (1- p_j)^2 L_j }{\tau_{t-1} p_j }, \, i ,j \in \{1, \ldots, m\}; t = 2, \ldots, k, \label{eta_tau_cond2}\\
\tfrac{\eta_k}{2} &\ge \tfrac{\tsum_{i=1}^m (p_i L_i)}{1+\tau_k},\label{eta_gamma_cond1}\\
\alpha_t  \theta_t &= \theta_{t-1}, t = 2, \ldots, k, \label{cond_alpha}
\end{align}
for some $\theta_t \ge 0$, $t =1, \ldots, k$.
Then, for any $k \ge 1$ and any given $z \in Z$, we have
\begin{align} 
\tsum_{t=1}^k \theta_t \bbe[Q((x^t, \hat y^t), z)]
&\le \eta_1 \theta_1  P(x^{0}, x) - (\mu+ \eta_k) \theta_k \bbe[P(x^k, x)] \nn\\
& \quad + \, \tsum_{i=1}^m \theta_1 \left(p_i^{-1} (1+\tau_1)-1\right) D_i(y^{0}_i, y_i).\label{ORIG_main}
\end{align}
\end{proposition}

\begin{proof}
Multiplying both sides of \eqnok{ORIG_rec} by $\theta_t$ and summing the resulting inequalities,
we have
\begin{align*}
\bbe[\tsum_{t=1}^k &\theta_t Q((x^t, \hat y^t), z)]  \le \bbe \left[\tsum_{t=1}^k \theta_t
\left( \eta_t P(x^{t-1}, x) - (\mu + \eta_t) P(x^t, x)
- \eta_t P(x^{t-1}, x^t)\right) \right] \\
& \quad + \tsum_{i=1}^m\bbe \left\{
\tsum_{t=1}^k \theta_t \left[ \left(p_i^{-1} (1+\tau_t)-1\right) D_i(y^{t-1}_i, y_i)- p_i^{-1} (1+ \tau_t) D_i(y^{t}_i, y_i)\right]\right\} \\
& \quad + \bbe\left[\tsum_{t=1}^k \theta_t \left( \langle \tilde x^t - x^t, U(\tilde y^t - y) \rangle
 - \tau_t p_{i_t}^{-1} D_{i_t}(y_{i_t}^{t-1}, y_{i_t}^t)\right)\right], 
\end{align*}
which, in view of the assumptions in \eqnok{cond_eta} and \eqnok{cond_tau}, then implies that
\begin{align}
\bbe[\tsum_{t=1}^k& \theta_t Q((x^t, \hat y^t), z)] 
\le  \eta_1 \theta_1 P(x^{0}, x) -  (\mu+ \eta_k) \theta_k \bbe[P(x^k, x)] \nn\\
& \quad + \, \tsum_{i=1}^m \left[ \theta_1 \left(p_i^{-1} (1+\tau_1)-1\right) D_i(y^{0}_i, y_i) 
-  p_i^{-1} \theta_k (1+ \tau_k) D_i(y^{k}_i, y_i)\right] \nn\\
& \quad - \, \bbe \left[ \tsum_{t=1}^k \theta_t \Delta_t\right],  \label{ORIG_sum}
\end{align}
where
\beq \label{ORIG_Delta}
\Delta_t := \eta_t P(x^{t-1}, x^t) - \langle \tilde x^t - x^t, U(\tilde y^t - y) \rangle
 + \tau_t p_{i_t}^{-1} D_{i_t}(y_{i_t}^{t-1}, y_{i_t}^t).
\eeq
We now provide a bound on $\tsum_{t=1}^k \theta_t \Delta_t$ in \eqnok{ORIG_sum}.
Note that by \eqnok{def_txt_r}, we have
\begin{align}
\langle \tilde x^t - x^t, U(\tilde y^t - y) \rangle &= 
\langle x^{t-1} - x^t, U(\tilde y^t - y) \rangle - \alpha_t \langle x^{t-2} - x^{t-1}, U(\tilde y^t - y) \rangle \nn\\
&=  \langle x^{t-1} - x^t, U(\tilde y^t - y) \rangle - \alpha_t \langle x^{t-2} - x^{t-1}, U(\tilde y^{t-1} - y) \rangle \nn\\
& \quad- \alpha_t \langle x^{t-2} - x^{t-1}, U(\tilde y^{t} - \tilde y^{t-1}) \rangle \nn \\
&= \langle x^{t-1} - x^t, U(\tilde y^t - y) \rangle - \alpha_t \langle x^{t-2} - x^{t-1}, U(\tilde y^{t-1} - y) \rangle \nn\\
& \quad- \alpha_t p_{i_t}^{-1} \langle x^{t-2} - x^{t-1}, y^{t}_{i_t} - y^{t-1}_{i_t} \rangle \nn\\
& \quad- \alpha_t (p_{i_{t-1}}^{-1}-1) \langle x^{t-2} - x^{t-1}, y^{t-2}_{i_{t-1}} - y^{t-1}_{i_{t-1}} \rangle, \label{bnd_inner_r}
\end{align}
where the last identity follows from the observation that by \eqnok{def_yt_r} and \eqnok{def_tyt_r},
\begin{align*}
U(\tilde y^{t} - \tilde y^{t-1})
& = \tsum_{i=1}^m \left\{\left[p_i^{-1} (y_i^t - y_i^{t-1}) + y_i^{t-1} \right] 
-\left[p_i^{-1} (y_i^{t-1} - y_i^{t-2}) + y_i^{t-2} \right]\right\}\\
&= \tsum_{i=1}^m \left\{\left[p_i^{-1} y_i^t - (p_i^{-1} -1) y_i^{t-1} \right] 
-\left[p_i^{-1} y_i^{t-1} - (p_i^{-1}-1) y_i^{t-2} \right]\right\}\\
&= \tsum_{i=1}^m \left[p_i^{-1} (y_i^t - y_i^{t-1})  + (p_i^{-1}-1)(y_i^{t-2}  - y_i^{t-1} ) \right] \\
&= p_{i_t}^{-1}(y^{t}_{i_t} - y^{t-1}_{i_t}) + (p_{i_{t-1}}^{-1}-1)(y^{t-2}_{i_{t-1}} - y^{t-1}_{i_{t-1}}).
\end{align*}
Using relation \eqnok{bnd_inner_r} in the definition of $\Delta_t$ in \eqnok{ORIG_Delta}, 
we have
\begin{align}
 \tsum_{t=1}^k \theta_t \Delta_t
&= \tsum_{t=1}^k \theta_t
\left[ \eta_t P(x^{t-1}, x^t)   \right. \nn\\
& \quad - \, \langle x^{t-1} - x^t, U(\tilde y^t - y) \rangle + \alpha_t \langle x^{t-2} - x^{t-1}, U(\tilde y^{t-1} - y) \rangle \nn\\
& \quad + \,  \alpha_t p_{i_t}^{-1} \langle x^{t-2} - x^{t-1}, y^{t}_{i_t} - y^{t-1}_{i_t} \rangle
 + \alpha_t (p_{i_{t-1}}^{-1}-1) \langle x^{t-2} - x^{t-1}, y^{t-2}_{i_{t-1}} - y^{t-1}_{i_{t-1}} \rangle \nn\\
&\quad + \, \left. p_{i_t}^{-1} \tau_t D_{i_t}(y^{t-1}_{i_t}, y^t_{i_t})\right]. \label{bound_Delta_r1}
\end{align}
Observe that by \eqnok{cond_alpha} and the fact that $x^{-1} = x^0$, 
\begin{align*}
&\tsum_{t=1}^k \theta_t \left[\langle x^{t-1} - x^t, U(\tilde y^t - y) \rangle - \alpha_t \langle x^{t-2} - x^{t-1}, U(\tilde y^{t-1} - y) \rangle\right]\\
&= \theta_k \langle x^{k-1} - x^k, U(\tilde y^k - y) \rangle \\
&=  \theta_k \langle x^{k-1} - x^k, U(y^k - y) \rangle 
+  \theta_k \langle x^{k-1} - x^k, U(\tilde y^k -y^k ) \rangle\\
&=  \theta_k \langle x^{k-1} - x^k, U(y^k - y) \rangle 
+  \theta_k (p_{i_k}^{-1} - 1) \langle x^{k-1} - x^k, y^k_{i_k} -y^{k-1}_{i_{k}}  \rangle,
\end{align*}
where the last identity follows from the definitions of $y^k$ and $\tilde y^k$ in \eqnok{def_yt_r} and  \eqnok{def_tyt_r},
respectively. Also, by the strong convexity of $P$ and $D_i$, we
have
\begin{align}
P(x^{t-1}, x^t)  \ge \tfrac{1}{2} \|x^{t-1} - x^t\|^2  \ \ \ \mbox{and} \ \ \
D_{i_t}(y^{t-1}_{i_t}, y^t_{i_t}) \ge \tfrac{1}{2 L_{i_t}} \|y^{t-1}_{i_t} - y^t_{i_t}\|^2.\nn
\end{align}
Using the previous three relations in \eqnok{bound_Delta_r1}, we have
\begin{align}
 \tsum_{t=1}^k \theta_t \Delta_t
&\ge \tsum_{t=1}^k \theta_t \left[ \tfrac{\eta_t }{2} \|x^{t-1} - x^t\|^2
+  \alpha_t p_{i_t}^{-1} \langle x^{t-2} - x^{t-1}, y^{t}_{i_t} - y^{t-1}_{i_t} \rangle
  \right. \nn \\
& \quad + \,  \left. \alpha_t (p_{i_{t-1}}^{-1}-1) \langle x^{t-2} - x^{t-1}, y^{t-2}_{i_{t-1}} - y^{t-1}_{i_{t-1}} \rangle 
+  \tfrac{\tau_t}{2 L_{i_t} p_{i_t}} \|y^{t-1}_{i_t} - y^t_{i_t}\|^2 \right] \nn\\
& \quad - \theta_k \langle x^{k-1} - x^k, U(y^k - y) \rangle 
-  \theta_k (p_{i_k}^{-1} - 1) \langle x^{k-1} - x^k, y^k_{i_k} -y^{k-1}_{i_{k}}  \rangle. \nn
\end{align}
Regrouping the terms in the above relation, and the fact that $x^{-1} = x^0$, we obtain
\begin{align}
 \tsum_{t=1}^k \theta_t \Delta_t 
&\ge \theta_k \left[\tfrac{\eta_k}{4} \|x^{k-1}- x^k\|^2 - \langle x^{k-1} - x^k, U(y^k - y) \rangle \right] \nn\\
&\quad + \theta_k 
\left[\tfrac{\eta_k}{4} \|x^{k-1}- x^k\|^2 - (p_{i_k}^{-1} - 1) \langle x^{k-1} - x^k, y^k_{i_k} -y^{k-1}_{i_{k}}  \rangle
+  \tfrac{\tau_k }{4 L_{i_k} p_{i_k}} \|y^{k-1}_{i_k} - y^k_{i_k}\|^2\right] \nn\\
& \quad + \, \tsum_{t=2}^k \theta_t \left[
\tfrac{\alpha_t}{ p_{i_t}} \langle x^{t-2} - x^{t-1},   y^{t}_{i_t} - y^{t-1}_{i_t} \rangle
 +  \tfrac{\tau_t }{4 L_{i_t} p_{i_t}} \|y^{t-1}_{i_t} - y^t_{i_t}\|^2  \right] \nn \\
& \quad + \, \tsum_{t=2}^k
\left[ \alpha_t \theta_t (p_{i_{t-1}}^{-1} - 1) \langle x^{t-2} - x^{t-1},   y^{t-2}_{i_{t-1}} - y^{t-1}_{i_{t-1}} \rangle
+  \tfrac{\tau_{t-1} \theta_{t-1} }{4 L_{i_{t-1}} p_{i_{t-1}}} \|y^{t-2}_{i_{t-1}} - y^{t-1}_{i_{t-1}}\|^2  \right]  \nn \\
& \quad + \, \tsum_{t=2}^k \tfrac{\theta_{t-1} \eta_{t-1} }{2} \|x^{t-2} - x^{t-1}\|^2 \nn \\
&\ge \theta_k \left[\tfrac{\eta_k}{4} \|x^{k-1}- x^k\|^2 - \langle x^{k-1} - x^k, U(y^k - y) \rangle \right] \nn\\
&\quad + \, \theta_k \left( \tfrac{\eta_k}{4} - \tfrac{L_{i_k} (1- p_{i_k})^2}{\tau_{k} p_{i_k}} \right) \|x^{k-1}- x^k\|^2 \nn\\
& \quad + \, \tsum_{t=2}^k \left[\tfrac{\theta_{t-1} \eta_{t-1}}{2} -  \tfrac{L_{i_t} \alpha_t^2 \theta_t}{\tau_t p_{i_t}}  
- \tfrac{\alpha_t^2 \theta_t^2 (1- p_{i_{t-1}})^2 L_{i_{t-1}} }{\tau_{t-1} \theta_{t-1} p_{i_{t-1}} }\right] \|x^{t-2}-x^{t-1}\|^2 \nn\\
&=  \theta_k \left[\tfrac{\eta_k}{4} \|x^{k-1}- x^k\|^2 - \langle x^{k-1} - x^k, U(y^k - y) \rangle \right] \nn\\
&\quad + \, \theta_k \left( \tfrac{\eta_k}{4} - \tfrac{L_{i_k} (1- p_{i_k})^2}{\tau_k p_{i_k}} \right) \|x^{k-1}- x^k\|^2 \nn\\
& \quad + \, \tsum_{t=2}^k \theta_{t-1} \left(\tfrac{ \eta_{t-1}}{2} - \tfrac{L_{i_t} \alpha_t }{\tau_t p_{i_t}}
 - \tfrac{(1- p_{i_{t-1}})^2 L_{i_{t-1}} }{\tau_{t-1} p_{i_{t-1}} }  \right) \|x^{t-2}-x^{t-1}\|^2  \nn \\
&\ge\theta_k \left[\tfrac{\eta_k}{4} \|x^{k-1}- x^k\|^2 - \langle x^{k-1} - x^k, U(y^k - y) \rangle \right], \label{bound_Delta_r2}
\end{align}
where the second inequality follows from the simple relation that 
\beq \label{simple}
b \langle u, v\rangle + a \|v\|^2/ 2 \ge - b^2 \|u\|^2 / (2a), \forall a > 0,
\eeq
and the last inequality follows from \eqnok{eta_tau_cond1} and
\eqnok{eta_tau_cond2}.
Plugging the bound \eqnok{bound_Delta_r2} into \eqnok{ORIG_sum}, we have
\begin{align}
\tsum_{t=1}^k& \theta_t \bbe[Q((x^t, \hat y^t), z)]
\le \theta_1 \eta_1 P(x^{0}, x) - \theta_k (\mu+ \eta_k) \bbe[P(x^k, x)] 
+ \tsum_{i=1}^m \theta_1 \left(p_i^{-1} (1+\tau_1)-1\right) D_i(y^{0}_i, y_i) \nn\\
& \quad - \, \theta_k \bbe \left[
\tfrac{\eta_k}{4} \|x^{k-1}- x^k\|^2 - \langle x^{k-1} - x^k, U(y^k - y) \rangle +
\tsum_{i=1}^m  p_i^{-1} (1+ \tau_k) D_i(y^{k}_i, y_i)  \right]. \nn
\end{align}
Also observe that by \eqnok{eta_gamma_cond1} and \eqnok{simple},
\begin{align*}
&\tfrac{\eta_k}{4} \|x^{k-1}- x^k\|^2 - \langle x^{k-1} - x^k, U(y^k - y) \rangle +
\tsum_{i=1}^m  p_i^{-1} (1+ \tau_k) D_i(y^{k}_i, y_i)  \\
&\ge \tfrac{ \eta_k}{4} \|x^{k-1}- x^k\|^2 + 
\tsum_{i=1}^m \left[ -\langle x^{k-1} - x^k, y^k_i - y_i \rangle +  \tfrac{1+\tau_k}{2 L_i p_i} \|y^k_i - y_i\|^2\right]  \\
&\ge \left( \tfrac{ \eta_k}{4} - \tfrac{\tsum_{i=1}^m (p_i L_i)}{2(1+\tau_k)} \right) \|x^{k-1} - x^k\|^2 \ge 0,
\end{align*}
The result then immediately follows by combining the above two conclusion.
\end{proof}

\vgap

\subsection{Proof of main convergence results}\label{sec_conver}
We now provide a proof for Theorem~\ref{cor_det1} which describes
the main convergence properties of the deterministic PDG method.

We first specialize Proposition~\ref{main_theorm} for the PDG method
applied to \eqnok{cpsaddle_f}.

\begin{proposition} \label{main_theorem_d}
Suppose
that $\{\tau_t\}$, $\{\eta_t\}$, and $\{\alpha_t\}$ in the PDG method satisfy
\begin{align}
\theta_t \tau_t &\le \theta_{t-1} (1+ \tau_{t-1}), t = 2, \ldots, k, \label{cond_d1}\\
\theta_t \eta_t &\le \theta_{t-1} (\mu + \eta_{t-1}), t = 2, \ldots, k, \label{cond_d2}\\
\eta_{t-1} \tau_{t} &\ge 2 L_f \alpha_t, t = 2, \ldots, k, \label{cond_d3}\\
\eta_k (1+\tau_k) &\ge 2 L_f, \label{cond_d4}\\
\alpha_t  &= \theta_{t-1} / \theta_t, t = 2, \ldots, k, \label{cond_d5} 
\end{align}
for some $\theta_t \ge 0$, $t = 1, \ldots, k$. Also let us denote $z^t=(x^t,g^t)$, and
\beq \label{def_avg_output}
\bar z^k := \left( \tsum_{t=1}^k \theta_t \right)^{-1} \tsum_{t=1}^k \theta_t z^t.
\eeq
Then, for any $k \ge 1$ and any given $(x, g) \in X \times {\cal G}$, we have
\begin{align} 
\label{PDG_main_d}
\left( \tsum_{t=1}^k \theta_t \right) Q_f(\bar z^k, z) + \theta_k(\mu+ \eta_k) P(x^k, x) 
&\le  \theta_1 \eta_1  P(x^{0}, x) + \theta_1 \tau_1 D_f(g^{0}, g).
\end{align}
\end{proposition}

\begin{proof}
Notice that in the deterministic PDG method, we have $m=1$, $p_i =1$, and $\hat y^t = g^t$.
It can be easily seen that the assumptions in \eqnok{cond_tau}-\eqnok{cond_alpha}
are implied by those in \eqnok{cond_d1}-\eqnok{cond_d5}. It then follows from \eqnok{ORIG_main}
that
\[
\tsum_{t=1}^k \theta_t Q_f(z^t, z) 
\le \theta_1 \eta_1  P(x^{0}, x) -  \theta_k(\mu+ \eta_k) P(x^k, x) + \theta_1 \tau_1 D_f(g^{0}, g). 
\]
Dividing both sides of the above inequality by  $ \tsum_{t=1}^k \theta_t $
and using the convexity of $Q(\bar z, z)$ w.r.t. $\bar z$, we have
\begin{align*}
 \left( \tsum_{t=1}^k \theta_t \right) Q_f(\bar z^k, z) 
 \le \tsum_{t=1}^k \theta_t Q_f(z^t, z)
&\le  \theta_1 \eta_1 P(x^0, x)  -  \theta_k(\mu+ \eta_k) P(x^k, x) 
+ \theta_1 \tau_1 D_f(g^{0}, g).
\end{align*}
Rearranging the terms in the above relation, we obtain \eqnok{PDG_main_d}.
\end{proof}

\vgap

We are now ready to show Theorem~\ref{cor_det1}.

\noindent{\bf Proof of Theorem~\ref{cor_det1}}
We first show part a).
It can be easily checked that \eqnok{cond_d1}-\eqnok{cond_d5} are satisfied with the selection of 
$\{\tau_t\}$, $\{\eta_t\}$, $\{\alpha_t\}$, and $\{\theta_t\}$ in \eqnok{constant_step_PDG1}. 
Using \eqnok{PDG_main_d} (with $x = x^*$ and $y = y^*$), \eqnok{bound_Dy0}, and 
the fact that $Q_f(\bar z, z^*) \ge 0$, we have
\[
 \theta_k(\mu+ \eta_k) P(x^k, x^*) 
\le \theta_1 (\eta_1 + L_f \tau_1) P(x^0, x^*), \ \ \forall k \ge 1.
\]
Using the parameter settings in \eqnok{constant_step_PDG1}, we conclude that
\begin{align*}
P(x^k,x^*) 
\le \tfrac{\theta_1 (\eta_1 + L_f \tau_1)}{\theta_k(\mu+ \eta_k)} P(x^0, x^*) 
=\tfrac{(\sqrt{2 L_f \mu} +  L_f \sqrt{2 L_f/\mu})}{\alpha (\mu + \sqrt{2 L_f\mu}) } \alpha^k  P(x^0, x^*)
= \tfrac{\mu + L_f}{\mu} \alpha^k P(x^0, x^*).
\end{align*}
Also using \eqnok{PDG_main_d} and the fact that $P(x^k, x) \ge 0$, we have
\beq \label{cor1_temp}
\left( \tsum_{t=1}^k \theta_t \right) Q_f(\bar z^k, z) 
\le  \theta_1 \eta_1  P(x^{0}, x) + \theta_1 \tau_1 D_f(g^{0}, g), \ \ \ \forall z \in Z.
\eeq
Denoting $\bar g^k_* := (\nabla f_1(\bar x^k); \ldots; \nabla f_m(\bar x^k))$,
we conclude from \eqnok{bound_Dy0} that
\begin{align*}
D_f(g^{0}, \bar g^k_*) 
&\le \tfrac{L_f}{2}\|\bar x^k - x^0\|^2 \le \tfrac{L_f}{2} [\tsum_{t=1}^k \theta_t]^{-1} \tsum_{t=1}^k \theta_t\|x^t - x^0\|^2 \\
&\le \tfrac{L_f}{2} [\tsum_{t=1}^k \theta_t]^{-1} \tsum_{t=1}^k \theta_t (\|x^t - x^*\|^2 + \|x^0 - x^*\|^2)\\
&\le \tfrac{L_f}{2} \left[ \tfrac{2(\mu + L_f)}{\mu} P(x^0, x^*) +  \|x^0 - x^*\|^2\right] \le L_f \left( \tfrac{2\mu + L_f}{\mu}\right) P(x^0, x^*),
\end{align*}
where the second inequality follows from the convexity of $\|\cdot\|^2$, the third inequality
follows from the triangular inequality, the fourth inequality follows from $\|x^t - x^*\|^2  \le 2 P(x^t, x^*)$
and \eqnok{PDG_cor1_result}, and the last inequality follows from $\|x^0 - x^*\|^2  \le 2 P(x^0, x^*)$.
Also note that by the definition of $\theta_t$, we have
\beq \label{sum_theta}
\tsum_{t=1}^k \theta_t = \tsum_{t=1}^k \alpha^{-t} = \tfrac{1-\alpha^k}{(1-\alpha) \alpha^k} \ge \tfrac{1}{\alpha^k},
\eeq
where the last inequality follows from the fact that $\alpha \le 1$ due to \eqnok{constant_step_PDG1}.
Fixing $g = \bar g^k_*$ in \eqnok{cor1_temp} and using the above two relations, we
obtain
\begin{align*}
Q_f(\bar z^k, (x, \bar g^k_*)) 
&\le \alpha^k \left[ \theta_1 \eta_1  P(x^{0}, x)  + L_f \theta_1 \tau_1
 \left( \tfrac{2\mu + L_f}{\mu}\right) P(x^0, x^*) \right] \\
&\le (\mu +  \sqrt{2L\mu}) \alpha^k \left[  P(x^{0}, x)  +  \tfrac{L_f}{\mu}(2+\tfrac{L_f}{\mu}) P(x^0,x^*) \right] \\
&= \tfrac{\mu \alpha^k}{1-\alpha} \left[  P(x^{0}, x)  +  \tfrac{L_f}{\mu}(2+\tfrac{L_f}{\mu}) P(x^0,x^*) \right].
\end{align*}
The result in \eqnok{PDG_cor1_result1} then directly follows from the above relation and \eqnok{PDgap1}.
If $X$ is bounded, the result in \eqnok{PDG_cor1_result2} then follows from the above relation, \eqnok{PDgap1},
and \eqnok{PDgap2}.

We now show part b). It is trivial to check that the conditions in \eqnok{cond_d1}-\eqnok{cond_d5}
hold by using our selection of $\{\tau_t\}$, $\{\eta_t\}$, $\{\alpha_t\}$, and $\{\theta_t\}$.
Using \eqnok{PDG_main_d} and the facts $\tau_1=0$ and $P(x^k,x) \ge 0$, we have
\[
\left( \tsum_{t=1}^k \theta_t \right) Q_f(\bar z^k, z) \le  \theta_1 \eta_1  P(x^{0}, x) = 4 L_f P(x^{0}, x).
\]
which, in view of \eqnok{def_pd_gap_unbounded} and \eqnok{PDgap1} and the fact that
$\tsum_{t=1}^k \theta_t = k(k+1)/2$,
clearly implies \eqnok{PDG_cor2_result}. In case $X$ is bounded,
the result in \eqnok{PDG_cor2_result1} immediately follows from \eqnok{PDgap1}, 
\eqnok{PDgap2}, and the above inequality.
\endproof

\vgap

We are now ready to provide a proof for Theorem~\ref{theorem_RPDG_s}, which 
describes the main convergence properties of the RPDG method applied to
strongly convex problems with $\mu > 0$.
\vgap

\noindent {\bf Proof of Theorem~\ref{theorem_RPDG_s}.}
It can be easily checked that the
conditions in \eqnok{cond_tau}-\eqnok{cond_alpha}
are satisfied with our requirements \eqnok{constant_step_RPDG}-\eqnok{cond_s3} of $\{\tau_t\}$, $\{\eta_t\}$, $\{\alpha_t\}$, and $\{\theta_t\}$.
Using the fact that $Q((x^t, \hat y^t), z^*) \ge 0$ , 
we then conclude from \eqnok{ORIG_main} (with $x= x^*$ and $y = y^*$)
that, for any $k\ge 1$, 
\begin{align*}
\bbe[P(x^k, x^*)] 
&\le \tfrac{1}{\theta_k(\mu+ \eta)} \left[\theta_1 \eta  P(x^{0}, x^*) + \tfrac{\theta_1 \alpha}{1-\alpha}  D(y^{0}, y^*)\right] 
\le  \left (1 + \tfrac{L_f \alpha}{(1-\alpha)\eta} \right) \alpha^k P(x^0, x^*), 
\end{align*}
where the first inequality follows from \eqnok{constant_step_RPDG} and \eqnok{cond_s1}, and the second inequality follows from \eqnok{cond_s2} and \eqnok{bound_Dy1}. 

Let us denote $\bar y^k \equiv (\tsum_{t=1}^k \theta_t)^{-1} \tsum_{t=1}^k (\theta_t \hat y^t)$,
$\bar z^k = (\bar x^k, \bar y^k)$. In view of \eqnok{psi_Q}, the convexity of $\|\cdot\|$, and \eqnok{strong_h}, we have
\begin{align}\label{rel1}
\bbe[\Psi(\bar x^k)-\Psi(x^*)]
&\le \bbe[Q(\bar{z}^k,z^*)]+\tfrac{L_f}{2}(\tsum_{t=1}^{k}\theta_t)^{-1}\bbe[\tsum_{t=1}^{k}\theta_t\|x^t-x^*\|^2]\nn\\
&\le \bbe[Q(\bar{z}^k,z^*)]+L_f(\tsum_{t=1}^{k}\theta_t)^{-1}\bbe[\tsum_{t=1}^{k}\theta_tP(x^t,x^*)].
\end{align}
Using \eqnok{ORIG_main} (with $x=x^*$ and $y=y^*$), the fact that $P(x^k, x) \ge 0$, and \eqnok{sum_theta}, we obtain
\begin{align*}
\bbe[Q(\bar z^k, z^*)] 
&\le \left( \tsum_{t=1}^k \theta_t \right)^{-1}\tsum_{t=1}^{k} \theta_t \bbe[Q((x^t,\hat y^t),z^*)]
\le \alpha^k \left (\alpha^{-1}\eta + \tfrac{L_f}{1-\alpha} \right) P(x^0, x^*).
\end{align*}
We conclude from \eqnok{RPDG_s_main} and the definition of $\{\theta_t\}$ that
\begin{align*}
(\tsum_{t=1}^{k}\theta_t)^{-1}\bbe[\tsum_{t=1}^{k}\theta_tP(x^t,x^*)]
& =(\tsum_{t=1}^k \alpha^{-t})^{-1}\tsum_{t=1}^k \alpha^{-t}(1+\tfrac{ L_f \alpha}{(1-\alpha)\eta})\alpha^{t} P(x^0, x^*)\\
& \le \tfrac{1-\alpha}{\alpha^{-k}-1}\tsum_{t=1}^{k}\tfrac{\alpha^{t}}{\alpha^{3t/2}}(1+\tfrac{ L_f \alpha}{(1-\alpha)\eta})P(x^0, x^*)\\
& = \tfrac{1-\alpha}{\alpha^{-k}-1}\tfrac{\alpha^{-k/2}-1}{1-\alpha^{1/2}}(1+\tfrac{ L_f \alpha}{(1-\alpha)\eta})P(x^0, x^*)\\
& = \tfrac{1+\alpha^{1/2}}{1+\alpha^{-k/2}}(1+\tfrac{ L_f \alpha}{(1-\alpha)\eta})P(x^0, x^*)
\le 2\alpha^{k/2}(1+\tfrac{ L_f \alpha}{(1-\alpha)\eta})P(x^0, x^*).
\end{align*} 

Using the above two relations, and \eqnok{rel1}, we obtain
\begin{align*}
\bbe[\Psi(\bar x^k) - \Psi(x^*)] 
&\le \alpha^k \left (\alpha^{-1}\eta + \tfrac{L_f}{1-\alpha} \right) P(x^0, x^*)
  +  L_f 2\alpha^{k/2}\left(1+\tfrac{ L_f \alpha}{(1-\alpha)\eta}\right)P(x^0, x^*)\\
&\le \alpha^{k/2}\left(\alpha^{-1}\eta + \tfrac{3-2\alpha}{1-\alpha}L_f+\tfrac{2 L_f^2 \alpha}{(1-\alpha)\eta}\right)P(x^0,x^*).
\end{align*}
\endproof

\subsection{Proof of the lower complexity bound}\label{sec_lower_bnd}
This subsection is devoted to the proof of Theorem~\ref{the_lower_bound}, which describes the performance limit for 
randomized incremental gradient methods. 

The following result provides an explicit expression for the optimal solution of
\eqnok{worst_problem}.

\begin{lemma}
Let $q$ be defined in \eqnok{def_q1}, $x^*_{i,j}$ is the $j$-th element of $x_i$, and
define 
\beq \label{opt_sol_worst}
x_{i,j}^* = q^j, i = 1, \ldots, m; j = 1, \ldots, \tilde n.
\eeq
Then
$x^*$ is the unique optimal solution of \eqnok{worst_problem}.
\end{lemma}

\begin{proof}
It can be easily seen that $q$ is the smallest root of the equation 
\beq \label{def_q}
q^2 - 2 \tfrac{\cQ + 1}{\cQ-1} q + 1 = 0.
\eeq
Note that
$x^*$ satisfies the optimality condition of \eqnok{worst_problem}, i.e.,
\beq \label{opt_cond_worst}
\left(A + \tfrac{4}{\cQ - 1} I \right)  x^*_i = e_1,  \ \ \ i = 1, \ldots, m.
\eeq
Indeed, we can write the coordinate form of \eqnok{opt_cond_worst} as
\begin{align}
 2 \tfrac{\cQ+1}{\cQ-1} x^*_{i,1} - x^*_{i,2} &= 1, \\
x^*_{i,j+1} - 2 \tfrac{\cQ+1}{\cQ-1} x^*_{i,j} + x^*_{i, j-1} &= 0, \ j = 2, 3, \ldots, \tilde n - 1,\\
- (\kappa + \tfrac{4}{\cQ -1}) x^*_{i, \tilde n} + x^*_{i, \tilde n -1} &= 0,
\end{align}
where the first two equations follow directly from the definition of $x^*$ and relation \eqnok{def_q},
and the last equation is implied by the definitions of $\kappa$ and $x^*$ in \eqnok{defA}
and \eqnok{opt_sol_worst}, respectively.
\end{proof}

We also need a few technical results to establish the lower complexity bounds.

\begin{lemma} \label{lemma_lower_bnd}
\begin{itemize}
\item [a)] For any $x > 1$, we have
\beq \label{bnd_log}
\log (1 - \tfrac{1}{x} ) \ge - \tfrac{1}{x-1}.
\eeq
\item [b)] Let $\rho, q, \bar q \in (0,1)$ be given.
If we have
\[
\tilde n \ge  \tfrac{t \log\bar q  + \log (1 - \rho)}{2 \log q},
\]
for any $ t \ge 0$, then 
\[
\bar q^t - q^{2 \tilde n} \ge \rho \bar q^t (1 - q^{2 \tilde n}).
\]
\end{itemize}
\end{lemma}

\begin{proof}
We first show part a).
Denote $\phi(x) = \log (1 - \tfrac{1}{x} ) + \tfrac{1}{x-1}$.
It can be easily seen that $\lim_{x \to +\infty} \phi(x) = 0$. Moreover,
for any $x > 1$, we have
\[
\phi'(x) = \tfrac{1}{x (x - 1)} - \tfrac{1}{(x-1)^2} = \tfrac{1}{x-1} \left( \tfrac{1}{x} - \tfrac{1}{x - 1}\right) < 0,
\]
which implies that $\phi$ is a strictly decreasing function for $x > 1$. Hence, we must have
$\phi(x) > 0$ for any $x > 1$.
Part b) follows from the following simple calculation.
\begin{align*}
\bar q^t - q^{2 \tilde n} - \rho \bar q^t (1 - q^{2 \tilde n})
= (1-\rho) \bar q^t  - q^{2 \tilde n} + \rho \bar q^t q^{2 \tilde n} \ge (1-\rho) \bar q^t  - q^{2 \tilde n} \ge 0.
\end{align*}
\end{proof}

We are now ready to prove Theorem~\ref{the_lower_bound}.

\noindent{\bf Proof of Theorem~\ref{the_lower_bound}}
Without loss of generality, we may assume that the initial point $x^0_i = 0$, $i = 1, \ldots, m$.
Indeed, the incremental gradient methods described in Subsection 3.3 are invariant with respect to a simultaneous
shift of the decision variables. In other words, the sequence of iterates $\{x^k\}$,
which is generated by such a method for minimizing the function $\Psi(x)$ 
starting from $x^0$, is just a shift of the sequence generated for minimizing 
$\bar \Psi(x) = \Psi(x + x^0)$
starting from the origin. 

Now let $k_i$, $i = 1, \ldots, m$, denote the number of times that the gradients of the 
component function $f_i$ are computed from iteration $1$ to $k$. Clearly
$k_i$'s are binomial random variables supported on $\{0, 1, \ldots, k\}$ such that
$\tsum_{i=1}^m k_i = k$. Also observe that
we must have $x^k_{i, j} = 0$ for any $k \ge 0$ and $k_j + 1 \le j \le \tilde n$, because each time the
gradient $\nabla f_i$ is computed, the incremental gradient methods add at most one more nonzero
entry to the $i$-th component of $x^k$ due to the structure of the gradient $\nabla f_i$.
Therefore, we have
\beq \label{bnd_dist_opt}
\tfrac{\|x^k - x^*\|^2_2}{\|x^0 - x^*\|^2_2}
= \tfrac{\tsum_{i=1}^m\|x_i^k - x_i^*\|^2_2}{\tsum_{i=1}^m \|x_i^*\|^2}
\ge \tfrac{\tsum_{i=1}^m \tsum_{j = k_i +1}^{\tilde n} (x^*_{i,j})^2 }{\tsum_{i=1}^m \tsum_{j=1}^{\tilde n} (x^*_{i,j})^2}
= \tfrac{\tsum_{i=1}^m (q^{2 k_i} - q^{2 \tilde n})}{m (1 - q^{2 \tilde n})}.
\eeq
Observing that for any $i = 1, \ldots, m$,
\[
\bbe[q^{2 k_i}] = \tsum_{t=0}^k \left[q^{2t} {k \choose t} p_i^t (1 - p_i)^{k-t}\right]
= [1 - (1 - q^2) p_i]^k,
\]
we then conclude from \eqnok{bnd_dist_opt} that
\[
\tfrac{\bbe[\|x^k - x^*\|^2_2]}{\|x^0 - x^*\|^2_2} \ge \tfrac{\tsum_{i=1}^m [1 - (1 - q^2) p_i]^k-mq^{2 \tilde n}}{m (1 - q^{2 \tilde n})}.
\]
Noting that $[1 - (1 - q^2) p_i]^k$ is convex w.r.t. $p_i$ for any $p_i \in [0,1]$ and $k \ge 1$, by
minimizing the RHS of the above bound w.r.t. $p_i$, $i = 1, \ldots, m$, subject to $\sum_{i=1}^m p_i = 1$ and
$p_i \ge 0$, we conclude that
\beq \label{bnd_dist_temp}
\tfrac{\bbe[\|x^k - x^*\|^2_2]}{\|x^0 - x^*\|^2_2} \ge \tfrac{ [1 - (1 - q^2) / m ]^k - q^{2 \tilde n}}{1 - q^{2 \tilde n}}
\ge \tfrac{1}{2}  [1 - (1 - q^2) / m ]^k,
\eeq
for any  $n \ge \underline n(m,k)$ (see \eqnok{def_N0}) and possible selection of $p_i$, $i = 1, \ldots, m$
satisfying \eqnok{prob_rand},
where the last inequality follows from Lemma~\ref{lemma_lower_bnd}.b).
Noting that
\begin{align*}
1 - (1 - q^2) / m = 1  - \left[ 1- \left(\tfrac{\sqrt{\cQ}-1}{\sqrt{\cQ}+1} \right)^2\right] \tfrac{1}{m}
= 1- \tfrac{1}{m} + \tfrac{1}{m} \left( 1 -\tfrac{2}{\sqrt{\cQ} + 1}\right)^2 \nn\\
= 1 - \tfrac{4}{m (\sqrt{\cQ} +1)} + \tfrac{4}{m (\sqrt{\cQ} +1)^2}
= 1- \tfrac{4 \sqrt{\cQ}}{m (\sqrt{\cQ} +1)^2},
\end{align*}
we then conclude from \eqnok{bnd_dist_temp} and Lemma~\ref{lemma_lower_bnd}.a) that
\begin{align*}
\tfrac{\bbe[\|x^k - x^*\|^2_2]}{\|x^0 - x^*\|^2_2} &\ge \tfrac{1}{2} \left[1- \tfrac{4 \sqrt{\cQ}}{m (\sqrt{\cQ} +1)^2}\right]^k
= \tfrac{1}{2} \exp\left(k \log  \left(1- \tfrac{4 \sqrt{\cQ}}{m (\sqrt{\cQ} +1)^2}\right) \right)\\
&\ge \tfrac{1}{2} \exp\left(- \tfrac{4 k \sqrt{\cQ} }{m (\sqrt{\cQ} +1)^2 - 4 \sqrt{\cQ}}  \right).
\end{align*}
\endproof
\section{Concluding remarks}\label{sec_remark}
In this paper, we present a new class of optimal first-order methods, referred to as primal-dual gradient methods,
for solving the finite-sum composite convex optimization problems given in the form of \eqnok{cp}. The optimal convergence of this algorithm
has been established based on the primal-dual optimality gap for the ergodic mean of iterates, i.e., $\bar z^k$, and the distance from 
the iterate $x^k$ to the optimal solution $x^*$. We also develop a randomized primal-dual gradient method which needs to compute
the gradient of only one randomly selected component $f_i$. The complexity bounds of the randomized primal-dual gradient method
have been
established in terms of the distance from the iterate $x^k$ to the optimal solution, and the primal optimality gap based on 
the ergodic mean of iterates, i.e., $\bbe[\Psi(\bar x^k) - \Psi^*]$. 
We show that these bounds are not improvable when the dimension $n$ is large enough by developing
new lower complexity bounds for randomized incremental gradient methods.
Extensions of the randomized primal-dual gradient method to non-strongly convex, nonsmooth, and unbounded problems
are also discussed in this paper. It should be noted that in this paper we focus on the theoretic convergence properties of
these primal-dual gradient methods, and the algorithmic parameters were chosen in a conservative manner and were
dependent on a few problem parameters, e.g., $L$ and $\mu$. In the future, it will be interesting
to develop more adaptive versions of these algorithms which do not require the explicit estimation about $L$ and $\mu$.


\bibliographystyle{plain}
\bibliography{../glan-bib}

\end{document}